\documentclass[11pt,reqno]{amsart}
\usepackage{amsfonts}
\usepackage{indentfirst}
\usepackage{graphicx}
\usepackage{amssymb}
\usepackage{color,xcolor}
\usepackage{graphicx}
\usepackage{epstopdf}
\usepackage{epsfig}
\usepackage{subfigure}
\usepackage{caption}
\usepackage{mathrsfs}
\usepackage{bbm}
\usepackage{chngpage}
\usepackage{cite}
\usepackage{multirow}
\usepackage{multicol}
\usepackage{dsfont}
\usepackage{bm}
\usepackage{amssymb,amsmath,amsthm,mathrsfs}
\usepackage{hyperref}
\usepackage[misc,geometry]{ifsym}
\allowdisplaybreaks

\hypersetup{hypertex=green,
            colorlinks=true,
            linkcolor=green,
            anchorcolor=green,
            citecolor=red}

\theoremstyle{definition}
\newtheorem{theorem}{Theorem}[section]
\newtheorem{definition}{Definition}[section]

\newtheorem{lemma}{Lemma}[section]

\numberwithin{equation}{section}%
\numberwithin{table}{section}%
\numberwithin{figure}{section}

\def\3bar{{|\hspace{-.02in}|\hspace{-.02in}|}}
\def\threeBar{{\big|\hspace{-.02in}\big|\hspace{-.02in}\big|}}
\def\td{\text{div}}
\def\tc{\text{curl}}
\def\d{\text{d}}

\begin{document}
\title{A priori and a posteriori error estimates for the quad-curl eigenvalue problem }

\author{Lixiu Wang}
\email{lxwang@csrc.ac.cn}
\address{Beijing Computational Science Research Center, Beijing, China}

\author{Qian Zhang(\Letter)}
\email{go9563@wayne.edu}
\address{Department of Mathematics, Wayne State University, Detroit, MI 48202, USA. }

\author{Jiguang Sun}
\email{jiguangs@csrc.ac.cn; jiguangs@mtu.edu}
\address{Department of Mathematical Sciences, Michigan Technological University, Houghton, MI 49931, USA}
\author{Zhimin Zhang}
\email{zzhang@math.wayne.edu;zmzhang@csrc.ac.cn}
\thanks{This work is supported in part by the National Natural Science Foundation of China grants NSFC 11871092, NSFC 11926356, and NSAF 1930402.}
\address{Beijing Computational Science Research Center, Beijing, China; Department of Mathematics, Wayne State University, Detroit, MI 48202, USA.}
\subjclass[2000]{subject class}
\date{\today}
\maketitle
\begin{abstract}
	In this paper, we propose a new family of $H(\tc^2)$-conforming elements for the quad-curl eigenvalue problem in 2D.
	The accuracy of this family is one order higher than that in \cite{WZZelement}. We prove a priori and a posteriori error estimates.
	The a priori estimate of the eigenvalue with a convergence order $2(s-1)$ is obtained if the eigenvector $\bm u\in \bm H^{s+1}(\Omega)$.
	For the a posteriori estimate, by analyzing the associated source problem,
	we obtain lower and upper bounds for the eigenvector in an energy norm and an upper bound for the eigenvalues.
		Numerical examples are presented for validation.
\end{abstract}
\section{Introduction}
The quad-curl equation appears in various applications, such as the inverse electromagnetic scattering theory \cite{Cakoni2010IP, Cakoni2017A, Sun2016A} or
magnetohydrodynamics \cite{Zheng2011A}. The corresponding quad-curl eigenvalue problem plays a fundamental role in the analysis and computation of the
electromagnetic interior transmission eigenvalues \cite{Monk2012Finite, sun2011iterative}.
To compute eigenvalues,  one usually starts with the corresponding source problem \cite{babuvska1991eigenvalue, boffi2010, Sun2016Finite}.
Some methods have been proposed for the source problem, i.e., the quad-curl problem,
in \cite{WZZelement, Zheng2011A,Sun2016A,Qingguo2012A,Brenner2017Hodge,quadcurlWG, Zhang2018M2NA,Chen2018Analysis164, Zhang2018Regular162,SunZ2018Multigrid102,WangC2019Anew101,BrennerSC2019Multigrid100}.
Recently, a family of $H(\tc^2)$-conforming finite elements using incomplete $k$-th order polynomials is proposed in \cite{WZZelement} for the qual-curl problem.
In this paper, we construct a new family of elements by using the complete $k$-th order polynomials.
Due to the large kernel space of the curl operator, the Helmholtz decomposition of splitting an arbitrary vector field into the irrotational and solenoidal components 
plays an important role in the analysis. However, in general, the irrotational component is not $H^2$-regular when $\Omega$ is non-convex. Therefore, we propose a new decomposition for $H_0(\tc^2;\Omega)$, which further splits the irrotational component into a function in $H^2(\Omega)$ and a function in the kernel space of curl operator.

There exist a few results on the numerical methods for the quad-curl eigenvalue problem.
The problem was first proposed in \cite{Sun2016A} by Sun, who applied a mixed finite element method and proved an a priori error estimate.
Two multigrid methods based on the Rayleigh quotient iteration and the inverse iteration with fixed shift were proposed and analyzed in \cite{han2020shifted}.
In the first part of the paper, we apply the classical framework of Babu\v{s}ka and Osborn \cite{babuvska1991eigenvalue,osborn1975spectral} to 
prove an a priori estimate. 

At reentrant corners or material interfaces, the eigenvectors feature strong singularities \cite{Nicaise2018Singularities161}.
For more efficient computation, adaptive local refinements are considered.
A posteriori error estimators are essential for the adaptive finite element methods.
We refer to \cite{cochez2007robust, beck2000residual,monk1998posteriori,schoberl2008posteriori} for the a posteriori estimates of source problems
and \cite{dai2008convergence,boffi2019posteriori,boffi2017residual} for eigenvalue problems.
In terms of the quad-curl eigenvalue problem, to the authors' knowledge, no work on a posteriori error estimations has been done so far.
To this end, we start by relating the eigenvalue problem to a source problem. An a posteriori error estimator for the source problem is constructed,
The proof uses the new decomposition and makes no additional regularity assumption.
Then we apply the idea of \cite{dai2008convergence} to obtain an a posteriori error estimate for the eigenvalue problem.


The rest of this paper is organized as follows. In Section 2, we present some notations, the new elements, the new decomposition,
and an $H(\tc^2)$ Cl\'ement interpolation. In Section~3, we derive an a priori error estimate for the quad-curl eigenvalue problem.
In Section 4, we prove an a posteriori error estimate. Finally, in Section 5, we show some numerical experiments.

\section{Notations and basis tools}
\subsection{Notations}Let $\Omega\in\mathbb{R}^2$ be a simply-connected Lipschitz domain.
For any subdomain $D\subset\Omega$, $L^2(D)$ denotes the space of square integrable functions on $D$ with norm $\|\cdot\|_D$. 
If $s$ is a positive integer,
$H^s(D)$ denotes the space of scalar functions in $L^2(D)$ whose derivatives up to order $s$ are also in $L^2(D)$. 
If $s=0$, $H^0(D)=L^2(D)$. When $D=\Omega$, we omit the subscript $\Omega$ in the notations of norms.
For vector functions, $\bm L^2(D) = (L^2(D))^2$ and $\bm H^s(D) = (H^s(D))^2$.

Let ${\bm u}=(u_1, u_2)^t$ and ${\bm w}=(w_1, w_2)^t$, where the superscript $t$ denotes the transpose.
Then ${\bm u} \times {\bm w} = u_1 w_2 - u_2 w_1$ and $\nabla \times {\bm u} = \partial u_2/\partial x_1 - \partial u_1/\partial x_2$.
For a scalar function $v$, $\nabla \times v = (\partial v/\partial x_2, - \partial v/\partial x_1)^t$.
We now define a space concerning the curl operator
\begin{align*}
H(\text{curl}^2;D)&:=\{\bm u \in {\bm L}^2(D):\; \nabla \times \bm u \in L^2(D),\;(\nabla \times)^2 \bm u \in \bm L^2(D)\},
\end{align*}
whose norm is given by
\[
\left\|\bm u\right\|_{H(\tc^2;D)}=\sqrt{(\bm u,\bm u)+(\nabla\times\bm u,\nabla\times\bm u)+((\nabla\times)^2\bm u,(\nabla\times)^2\bm u)}.
\]
The spaces $H_0(\text{curl}^2;D)$, $H_0^1(D)$, and $H(\text{div}^0;D)$ are defined, respectively, as
\begin{align*}
&H_0(\text{curl}^2;D):=\{\bm u \in H(\text{curl}^2;D):\;{\bm n}\times\bm u=0\; \text{and}\; \nabla\times \bm u=0\;\; \text{on}\ \partial D\},\\
&H_0^1(D):=\{u \in H^1(D):u=0\;\; \text{on}\ \partial D\},\\
&H(\text{div}^0;D) :=\{\bm u\in {\bm L}^2(D):\; \nabla\cdot \bm u=0\}.
\end{align*}

Let \,$\mathcal{T}_h\,$ be a triangular partition of $\Omega$.
Denote by $\mathcal{N}_h$ and $\mathcal{E}_h$ the sets of vertices and  edges. Let ${\bm \tau}_e$ be the tangent vector of an edge $e \in \mathcal{E}_h$.
We refer to $\mathcal{N}_h^{\text{int}}$ and  $\mathcal{E}_h^{\text{int}}$  as the sets of vertices and edges in the interior of $\Omega$, respectively.
 Let $\mathcal{N}_h(T)$ and $\mathcal{E}_h(T)$ be the sets of vertices and edges on the element $T$. Denote by $h_T$ the diameter of
 $T \in \mathcal{T}_h$ and $\displaystyle h = \max_{T\in \mathcal {T}_h}h_T$. In the following,
we introduce some  subdomains called patches:
\begin{itemize}
  \item $\omega_T$: the union of elements sharing a common edge with $T$, $T\in\mathcal{T}_h$;
  \item $\omega_e$: the union of elements sharing $e$ as an edge, $e\in\mathcal{E}_h$;
  \item $\omega_v$: the union of elements sharing $v$ as  a vertex, $v\in \mathcal{N}_h$.
\end{itemize}
We use $P_k$ to represent the space of polynomials on an edge or on a subdomain $D\subset\Omega$ with degrees at most $k$ and $\bm P_k(D)=\left(P_k(D)\right)^2$.



\subsection{A decomposition of $H_0(\tc^2;\Omega)$ }
We mimic the proof of \cite[Prop. 5.1]{dhia1999singular} to obtain a decomposition of the space $H_0(\tc^2;\Omega)$, which plays a critical role in the analysis.
\begin{lemma}\label{Helm}
	Let $\nabla H_0^1(\Omega)$ be the set of gradients of functions in $H_0^1(\Omega)$.
	Then $\nabla H_0^1(\Omega)$ is a closed subspace of $H_0(\mathrm{curl}^2;\Omega)$ and
	\begin{align}\label{decom-00}
H_0(\mathrm{curl}^2;\Omega)=X\oplus \nabla H_0^1(\Omega),
\end{align}
where $X=\left\{\bm u\in H_0(\tc^2;\Omega)\big|(\bm u,\nabla p)=0,\;\;\forall p\in H_0^1(\Omega)\right\}.$
Namely, for $\bm u\in H_0(\mathrm{curl}^2;\Omega)$,
$\bm u=\bm u^0+\bm u^{\perp}$
with $\bm u^0\in \nabla H_0^1(\Omega)$ and $\bm u^{\perp}\in X.$
Furthermore, $\bm u^{\perp}$ admits the splitting
\begin{align}\label{decom-01}
\bm u^{\perp}=\nabla \phi+\bm v,
\end{align}
where $\phi\in H_0^1(\Omega)$ and $\bm v \in \bm H^2(\Omega)$ satisfying
\begin{align}
&\|\bm v\|_2\leq C\|\nabla\times\bm u^{\perp}\|_1.\label{decom-02}\\
&\|\nabla\phi\|\leq C\left(\|\nabla\times\bm u^{\perp}\|_1+\|\bm u^{\perp}\|\right).\label{decom-03}
\end{align}
\end{lemma}
\begin{proof}
The proof of \eqref{decom-00} can be found in \cite{WZZelement}. We only need to prove \eqref{decom-01}.
Let $\mathcal{O}$ be a bounded, smooth, simply-connected open set with $\bar{\Omega}\subset\mathcal{O}$.
For any $\bm u^{\perp}\in X$, we can extend $\bm u^{\perp}$ in the following way:
\begin{align*}
\bm {\tilde{u}}&=\begin{cases}
\bm u^{\perp},&\Omega,\\
0,&\mathcal{O}-\bar{\Omega}.\\
\end{cases}
\end{align*}
Obviously, $\bm {\tilde{u}}\in H_0(\tc^2;\mathcal{O})$ and $\nabla \times \bm {\tilde{u}}\in H_0^1(\mathcal{O})$.
Now, we consider the following problem:
Find $\psi$ defined in $\mathcal{O}$ such that
\begin{align}\label{lap-01}
-\triangle \psi&=\nabla \times \bm {\tilde{u}},\ \text{in}\ \mathcal{O},\\
\psi&=0,\ \text{on}\ \partial \mathcal{O}.\label{lap-02}
\end{align}
Since $\nabla \times \bm {\tilde{u}}\in H_0^1(\mathcal{O})$ and $\mathcal O$ has a smooth boundary, there exists a function $\psi\in H^3(\mathcal{O})$ satisfying \eqref{lap-01} and \eqref{lap-02} and
\begin{align}\label{reg1}
	\left\|\psi\right\|_{3,\mathcal O}\leq C\|\nabla \times \bm {\tilde{u}}\|_{1,\mathcal O}.
\end{align}
In addition, \eqref{lap-01} can be rewritten as
$
\nabla\times(\nabla\times \psi-\bm {\tilde{u}})=0.
$
Based on \cite[Thm. 2.9]{Girault2012Finite}, there exists a unique function $p$ of $H^1(\mathcal{O})/\mathbb{R}$ such that
\begin{align}\label{proof-lem}
-\nabla\times \psi+\bm {\tilde{u}}=\nabla p.
\end{align}
Now, we restrict \eqref{proof-lem} to the domain $\mathcal{O}-\bar{\Omega}$ and obtain
\begin{align}
\nabla p=-\nabla\times \psi\in H^2(\mathcal{O}/\bar{\Omega}).
\end{align}
Using the extension theorem \cite{chen2016sobolev}, we can extend $p\in H^3(\mathcal{O}/\bar{\Omega})$ to $\tilde{p}$ defined on $\mathcal O$ satisfying
\begin{align}\label{reg2}
	\left\|\tilde{p}\right\|_{3,\mathcal O}\leq C\left\|{p}\right\|_{3,\mathcal O\slash \bar{\Omega}}\leq C\left\|{\nabla p}\right\|_{2,\mathcal O\slash \bar{\Omega}}\leq C\left\|{\nabla \times \psi}\right\|_{2,\mathcal O\slash \bar{\Omega}},
\end{align}
where we have used Poincar\'e-Friedrichs inequality for $p\in H^3(\mathcal O\slash\bar{\Omega})$ since we can choose $p$ for which $\int_{\mathcal O\slash\bar{\Omega}}p=0.$ 
Restricting on $\Omega$, we have
\begin{align*}
\bm {{u}^{\perp}}=\underbrace{\nabla\times\psi+\nabla \tilde{p}}_{\in H^2(\Omega)}+\nabla\underbrace{(p-\tilde{p})}_{\in H^1(\Omega)}
\triangleq\bm v+\nabla \phi.
\end{align*}
Note that $\phi=p-\tilde p\in H_0^1(\Omega)$ since $\tilde p$ is the extension of $p$.
Therefore, \eqref{decom-01} is proved.
Combining \eqref{reg1} and \eqref{reg2}, we obtain
\[\|\bm v\|_{2,\Omega}=\|\nabla\times \psi+\nabla \tilde{p}\|_{2,\Omega}\leq\|\nabla\times \psi+\nabla \tilde{p}\|_{2,\mathcal{O}}\leq C\|\nabla\times\psi\|_{2,\mathcal{O}}\leq C\|\nabla\times\bm {\tilde u}\|_{1,\mathcal{O}}=C\|\nabla\times\bm { u}^{\perp}\|_{1,\Omega}\]
and
\[\|\nabla \phi\|_{\Omega}=\|\bm { u}^{\perp}-\bm v\|_{\Omega}\leq\|\bm { u}^{\perp}\|_{\Omega}+\left\|\bm v\right\|_{\Omega}\leq \|\bm { u}^{\perp}\|_{\Omega}+\left\|\bm v\right\|_{2,\Omega}\leq C\left(\|\bm { u}^{\perp}\|_{\Omega}+\|\nabla\times\bm u^{\perp}\|_{1,\Omega}\right).\]
\end{proof}

\subsection{A new family of $H(\tc^2)$-conforming elements}
In this subsection, we propose a new family of $H(\tc^2)$-conforming finite elements. 
The new elements can lead to one order higher accuracy than the elements in \cite{WZZelement} when the solution $\bm u$ is smooth enough.
\begin{definition}\label{tri-dof-def}
	For an integer $k \geq 4$, an ${H}(\mathrm{curl}^2)$-conforming element is given by the triple:
	\begin{equation*}
	\begin{split}
	&{T}\;\text{is a triangle},\\
	&P_{T} = \bm P_k(T),\\
	&\Sigma_{{T}} = \bm M_{{p}}( {\bm u }) \cup \bm M_{{e}}({\bm u}) \cup \bm M_{K}({\bm u}),
	\end{split}
	\end{equation*}
	where $\Sigma_{{T}}$ is the set of DOFs (degree of freedom) defined as follows.
	\begin{itemize}
		\item $\bm M_{{p}}({\bm u})$ is the set of DOFs on all vertex nodes and edge nodes ${p}_{i}$:
		\begin{equation}\label{2def2}
	\bm M_{{p}}( {\bm u})=\left\{{\nabla}\times {\bm u}({p_{i}}), i=1,\;2,\;\cdots\;,3k\right\}
	\end{equation}
	with the points $p_{i}$ chosen at $3$ vertex nodes and $(k-1)$ distinct nodes on each edge.
	\item $\bm M_{{e}}({\bm u})$ is the set of DOFs given on all edges ${e}_i$ of ${T}$ with the unit  tangential vector ${\bm \tau}_{e_i}$:
	 \begin{equation}\label{2def3}
	\bm M_{{e}}( {\bm u})= \left\{\int_{{e}_i} {\bm u}\cdot {\bm \tau}_{e_i} {q}\mathrm d {s},\ \forall   {q}\in P_{k}({e}_i),\;i=1,2,3 \right\}.
	\end{equation}
	\item $\bm M_{T}( {\bm u})$ is the set of DOFs on the element ${T}$:
	\begin{align}\label{2def4}
	\bm M_{{T}}({\bm u})=\left\{\int_{T} {\bm u}\cdot {\bm q} \mathrm d  V,\  \forall \bm q \in \mathcal{D}\right\},
	\end{align}
	where $\mathcal{D}=\bm P_{k-5}( T)\oplus\widetilde{P}_{k-5}{\bm x}\oplus\widetilde{P}_{k-4}{\bm x}\oplus
	\widetilde{P}_{k-3}{\bm x}\oplus
	\widetilde{P}_{k-2}{\bm x}$ when $k\geq 5$ and $\mathcal{D}=\widetilde{P}_{0}{\bm x}\oplus
	\widetilde{P}_{1}{\bm x}\oplus
	\widetilde{P}_{2}{\bm x}$ when $k=4$.  Here  $\widetilde{P}_k$ is the space of a homogeneous polynomial of degree $k$.  
	\end{itemize}
	
\end{definition}
%
%
\begin{lemma}\label{unisolvence}
	The above finite elements are unisolvent and $H(\mathrm{curl}^2)$-conforming.
\end{lemma}

Using the above Lemma, the global finite element space $V_h$ on $\mathcal{T}_h$ is given by \begin{align*}
	&V_h=\{\bm{v}_h\in H(\text{curl}^2;\Omega):\ \bm v_h|_T\in \bm {P}_k(T)\ \forall T\in\mathcal{T}_h\}.
\end{align*}
Provided $\bm u \in \bm H^{1/2+\delta}(\Omega)$ and $ \nabla \times \bm u \in H^{1+\delta}(\Omega)$ with $\delta >0$,  
define an $H(\tc^2;\Omega)$ interpolation $\Pi_h\bm u\in V_h$, whose restriction on $T$, denoted by $\Pi_T\bm u$, is such that
\begin{eqnarray}\label{def-inte-tri}
	\bm M_p(\bm u-\Pi_T\bm u)=0,\ \bm M_e(\bm u-\Pi_T\bm u)=0,\ \text{and}\ \bm M_T(\bm u-\Pi_T\bm u)=0,
\end{eqnarray}
where $\bm M_p,\ \bm M_e$, and $\bm M_T$ are the sets of DOFs in \eqref{2def2}-\eqref{2def4}.

\begin{theorem}\label{err-interp}
	If $\bm u\in \bm H^{s+1}(\Omega)$, $1+\delta\leq s\leq k$ with $\delta>0$, then the following error estimate for the interpolation $\Pi_h$ holds:
	\begin{align*}
	&\left\|\bm u-\Pi_h\bm u\right\|_T+h_T\left\|\nabla\times(\bm u-\Pi_h\bm u)\right\|_T+h_T^2\left\|(\nabla\times)^2(\bm u-\Pi_h\bm u)\right\|_T\leq C{h^{s+1}}\left\|\bm u\right\|_{s+1,T}.
	\end{align*}
\end{theorem}
Since the proofs of Lemma~\ref{unisolvence} and Theorem~\ref{err-interp} are similar to those in \cite{WZZelement}, we omit them.


\subsection{An $H$(curl$^2$)-type Cl\'ement interpolation}

 Let $\omega_{v}$ be a patch on a vertex $v$ and $R^k_{v}\phi$ be the $L^2$ projection of $\phi$ on $\omega_{v}$, i.e., $R_{v}^k\phi\in P_k(\omega_{v})$ such that
 \begin{align*}
 \int_{\omega_{v}}\left(\phi-R_{v}^k\phi\right)p\d V=0, \quad \forall p\in P_k(\omega_{v}).
 \end{align*}
Similarly, we can define an $L^2$ projection $R_e^k$ on a patch $\omega_e$ for an edge $e$.

For $\bm u\in H_0(\tc^2;\Omega)$, the lowest-order $H(\tc^2;\Omega)$ interpolation $\Pi_h\bm u$ can be rewritten as
\begin{align*}
\Pi_h\bm u=\sum_{v\in \mathcal{N}_h^{\text{int}}}\alpha_v(\bm u)\bm \phi_v+\sum_{e\in \mathcal{E}_h^{\text{int}}}\sum_i\big(\alpha_{e}^{\text{curl},i}(\bm u)\bm \phi_{e}^{\text{curl},i}+\alpha_{e}^i(\bm u)\bm \phi_{e}^i\big)+\sum_{T\in\mathcal{T}_h}\sum_{i}\alpha_{T}^i(\bm u)\bm \phi_{T}^i,
\end{align*}
where 
\begin{align*}
\alpha_v(\bm u)&=\nabla\times \bm u(v)\text{ for any vertex }v,\\
\alpha_{e}^{\text{curl},i}(\bm u)&=\nabla\times \bm u(v_{e,i})\text{ for any node $v_{e,i}$ on an edge }e,\\
\alpha_{e}^i(\bm u)&=\int_{e} {\bm u}\cdot{\bm \tau}_e {q_i}\d{s}\ \text{for any } {q_i}\in P_{4}({e}),\\
\alpha_{T}^i(\bm u)&=\int_{{T}} {\bm u}\cdot{\bm q}_i\d V \text{  for any }{\bm q}_i\in \mathcal{D},
\end{align*}
and the functions $\bm\phi_v$, $\bm\phi_{v_e}$, $\bm\phi_e^i$, and $\bm\phi_T^i$ are the corresponding Lagrange basis functions.
Now we define a new $H(\tc^2)$ Cl\'ement interpolation $\Pi_{C}$ for $\bm u\in \{\bm u\in\bm H^{1/2+\delta}(\Omega)|\ \nabla\times\bm u \in H^{1}(\Omega)\}$:
\begin{align*}
\Pi_{C}\bm u=\sum_{v\in \mathcal{N}_h^{\text{int}}}\tilde{\alpha}_v(\bm u)\bm \phi_v+\sum_{e\in \mathcal{E}_h^{\text{int}}}\sum_i\big(\tilde{\alpha}_{e}^{\text{curl},i}(\bm u)\bm \phi_{e}^{\text{curl},i}+\alpha_{e}^i(\bm u)\bm \phi_{e}^i\big)+\sum_{T\in\mathcal{T}_h}\sum_{i}\alpha_{T}^i(\bm u)\bm \phi_{T}^i,
\end{align*}
where $\tilde{\alpha}_v(\bm u)=R^4_{v}(\nabla\times \bm u)(v)$ and $\tilde{\alpha}_{e}^{\text{curl},i}(\bm u)= R^4_{e}(\nabla\times \bm u)(v_{e,i})$.
The interpolation is well-defined and the following error estimate holds.
\begin{theorem}\label{clmt-error}
For any $T\in \mathcal T_h$, let $\omega=\{\omega_{v_i}: T\subset \omega_{v_i}\}$. Then, for $\bm u \in \bm H^2(\Omega)$, it holds that
\begin{align}
&\|\bm u-\Pi_{C} \bm u\|_T+h_T\|\nabla(\bm u-\Pi_{C} \bm u)\|_T+h_T^2\|\nabla\left(\nabla\times(\bm u-\Pi_{C} \bm u)\right)\|_T\le C h^{2}\|\bm u\|_{2,\omega}.\label{err-05}
\end{align}
\end{theorem}
The theorem can be obtained using the similar arguments for Theorem \ref{err-interp} and the boundedness of the operators $R_v^4,\ R_e^4$.

%

\section{An a priori error estimate for the eigenvalue problem}
Following \cite{Sun2016A}, the quad-curl eigenvalue problem is to seek $\lambda$ and $\bm u$ such that
\begin{equation}\label{eig-prob1}
\begin{split}
(\nabla\times)^4\bm u&=\lambda\bm u\ \text{in}\;\Omega,\\
\nabla \cdot \bm u &= 0\ \text{in}\;\Omega,\\
\bm u\times\bm n&=0\;\text{on}\;\partial \Omega,\\
\nabla \times \bm u&=0\ \text{on}\;\partial \Omega,
\end{split}
\end{equation}
where $\bm n$ is the unit outward normal to $\partial \Omega$. The assumption that $\Omega$ is simply-connected implies $\lambda\neq 0$.
The variational form of the quad-curl eigenvalue problem is to find $\lambda\in \mathbb{R}$ and $\bm u \in X$ such that
\begin{align}\label{eig-02}
((\nabla\times)^2\bm u,(\nabla\times)^2\bm v)=\lambda(\bm u,\bm v),\quad \forall \bm v \in X.
\end{align}
In addition to $V_h$ defined in Section 3, we need more discrete spaces. Define
\begin{eqnarray*}
	&&   V^0_h=\{\bm{v}_h\in V_h:\ \bm{n} \times \bm{v}_h=0\ \text{and}\ \nabla\times  \bm{v}_h = 0 \ \text {on} \ \partial\Omega\},\\
	&&  S_h=\{{w}_h\in H^1(\Omega):\  w_h|_T\in P_k\},\\
	&&  S^0_h=\{{w}_h\in W_h,\;{w}_h|_{\partial\Omega}=0\},\\
	&& X_h=\{\bm u_h\in V^0_h \ | \ (\bm u_h,\nabla q_h)=0,\ \  \text{for all}\ \ q_h\in S^0_h\}.
\end{eqnarray*}
The discrete problem for \eqref{eig-02} is to find $\lambda_h \in \mathbb R$ and  $\bm u_h\in X_h$ such that
\begin{equation}\label{eig-dis-01}
\begin{split}
((\nabla\times)^2\bm u_h,(\nabla\times)^2\bm v_h) &=\lambda_h(\bm u_h, \bm v_h),\quad \forall \bm v\in X_h.
\end{split}
\end{equation}
\subsection{The source problem}
We start with the associated source problem.
Given $\bm f\in L^2(\Omega)$, find $\bm u\in H_0(\tc^2;\Omega)$ and $p\in H_0^1(\Omega)$ such that
\begin{equation}\label{prob1}
\begin{split}
(\nabla\times)^4\bm u+\bm u+\nabla p&=\bm f\ \text{in}\;\Omega,\\
\nabla \cdot \bm u &= 0\ \text{in}\;\Omega,\\
\bm u\times\bm n&=0\;\text{on}\;\partial \Omega,\\
\nabla \times \bm u&=0\ \text{on}\;\partial \Omega.
\end{split}
\end{equation}
Note that $p=0$ for $\bm f\in H(\td^0;\Omega)$.

The weak formulation is to find $(\bm u;p)\in H_0(\tc^2;\Omega)\times H_0^1(\Omega)$ such that
\begin{equation}\label{prob22}
\begin{split}
a(\bm u,\bm v) + b(\bm v, p)&=(\bm f, \bm v),\quad \forall \bm v\in H_0(\tc^2;\Omega),\\
b(\bm u,q)&=0,\quad \hspace{0.7cm}\forall q\in H^1_0(\Omega),
\end{split}
\end{equation}
where
\begin{align*}
	a(\bm u,\bm v)&=\left((\nabla\times)^2 \bm u, (\nabla\times)^2 \bm v\right)+(\bm u,\bm v),\\
b(\bm v,p)&=(\bm v,\nabla p ).
\end{align*}

The well-posedness of \eqref{prob22} is proved in Thm. 1.3.2 of \cite{Sun2016Finite}.
Consequently, we can define an solution operator $A: \bm L^2(\Omega)\rightarrow \bm L^2(\Omega)$ such that ${\bm u} = A{\bm f} \in X \subset  L^2(\Omega)$. 
In fact,  $A$ is compact due to the following result.
\begin{lemma}\label{continous-cmpct}
	$X$ processes the continuous compactness property.
\end{lemma}
\begin{proof}
	Since $X\subset Y:=\left\{\bm u\in H_0(\tc;\Omega)\big|(\bm u,\nabla p)=0,\;\;\forall p\in H_0^1(\Omega)\right\}\hookrightarrow\hookrightarrow L^2(\Omega)$ \cite{Monk2003}, then $X\hookrightarrow\hookrightarrow L^2(\Omega)$.
\end{proof}


The $H(\tc^2)$-conforming FEM seeks $\bm u_h\in V^0_h$ and $p_h\in S_h^0$  such that
\begin{equation}\label{prob3}
\begin{split}
a(\bm u_h,\bm v_h)+b(\bm v_h,p_h) &=(\bm f, \bm v_h),\quad \forall \bm v_h\in V^0_h,\\
b(\bm u_h,q_h)&=0,\quad \hspace{0.9cm}\forall q_h\in S^0_h.
\end{split}
\end{equation}
The well-posedness of problems \eqref{prob3} is due to the discrete compactness of $\{X_h\}_{h\in \Lambda}$ with $\Lambda={h_n,n=0,1,2,\cdots}$,
which is stated in the following theorem. Its proof is similar to that of Theorem 7.17 in \cite{Monk2003} and thus is omitted.

%
%

\begin{theorem}\label{discrete-compact}
	$X_h$ processes the discrete compactness property.
\end{theorem}

Consequently, we can define a discrete solution operator $A_h: \bm L^2(\Omega)\rightarrow \bm L^2(\Omega)$
such that $\bm u_h = A_h\bm f$ is the solution of \eqref{prob3}. It is straightforward to use the standard finite element framework and the approximation property of the interpolation to show the following theorem.

\begin{theorem}\label{conv-curlcurl}
	Assume that $A\bm f\in \bm H^s(\Omega)$, $\nabla\times A\bm f\in H^s(\Omega)$ $(1+\delta\leq s\leq k\;\text{with}\;\delta>0)$, and $p\in H^s(\Omega)$. It holds that
	\begin{equation*}
	\begin{split}
	&\|A\bm f-A_h\bm f\|_{H(\mathrm{curl}^2;\Omega)} \leq C h^{s-1}\left(\left\|A\bm f\right\|_s+\left\|\nabla\times A\bm f\right\|_s+\left\|p\right\|_s\right).
	\end{split}
	\end{equation*}
\end{theorem}

\subsection{An a priori error estimate of the eigenvalue problem}
We first rewrite the eigenvalue problem as follows. Find $\lambda\in \mathbb{R}$ and $(\bm u;p)\in H_0(\tc^2;\Omega)\times H_0^1(\Omega)$ such that
\begin{equation}\label{eig-01}
\begin{split}
a(\bm u,\bm v) + b(\bm v,p)&=(\lambda+1)(\bm u, \bm v),\quad \forall \bm v\in H_0(\tc^2;\Omega),\\
b(\bm u,q)&=0,\quad \hspace{0.7cm}\forall q\in H^1_0(\Omega).
\end{split}
\end{equation}
Due to the fact that $\nabla\cdot \bm u=0$, we have $p=0$.
Then \eqref{eig-01} can be written as an operator eigenvalue problem of finding $\mu:= 1/(\lambda+1) \in \mathbb{R}$ and $\bm u \in X$ such that
\begin{align}\label{eig-03}
A \bm u=\mu\bm u.
\end{align}

The discrete eigenvalue problem is to find $\lambda_h\in \mathbb{R}$ and $(\bm u_h; p_h)\in V_h^0\times S_h^0$ such that
\begin{equation}\label{eig-dis-001}
\begin{split}
a(\bm u_h,\bm v_h) + b(\bm v_h,p_h)&=(\lambda_h+1)(\bm u_h, \bm v_h),\quad \forall \bm v_h\in V_h^0,\\
b(\bm u_h,q_h)&=0,\quad \hspace{0.7cm}\forall q_h\in S^0_h.
\end{split}
\end{equation}
Using the operator $A_h$, the eigenvalue problem is to find $\mu_h\in \mathbb{R}$ and $\bm u_h\in X_h$ such that
\begin{align}\label{eig-dis-02}
A_h \bm u_h=\mu_h\bm u_h,
\end{align}
where $\mu_h = 1/(\lambda_h+1)$.

Define a collection of operators,
\[\mathcal A=\{A_h:\bm L^2(\Omega)\rightarrow \bm L^2(\Omega), \ h\in \Lambda\}.\]
Due to Theorem \ref{discrete-compact} and Theorem \ref{conv-curlcurl}, 

(1) \(\mathcal{A}\) is collectively compact, and

(2) \(\mathcal{A}\) is point-wise convergent, i.e., for \(\boldsymbol{f} \in\bm L^{2}(\Omega),\)
$
\begin{array}{l}{A_{h_{n}} \boldsymbol{f} \rightarrow A \boldsymbol{f}} \end{array}
$
 strongly in $\bm L^{2}(\Omega)$  as   $n \rightarrow \infty$.

\begin{theorem}\label{eigenvalue-convergence}
Let $\mu$ be an eigenvalue of $A$ with multiplicity $m$ and $E(\mu)$ be the associated eigenspace. Let $\{\bm \phi_j\}_{j=1}^m$ be an orthonormal basis for $E(\mu)$. 
Assume that  $\bm\phi\in \bm H^{s+1}(\Omega)$ for $\bm\phi\in E(\mu)$.
Then, for $h$ small enough, there exist exactly $m$ discrete eigenvalues $\mu_{j,h}$ and the associated eigenfunctions $\bm \phi_{j,h}, j=1,2,\cdots,m,$ of $A_h$ such that
\begin{align}
|\mu-\mu_{j,h}|&\leq C\max_{1 \leq i\leq m}a(\bm \phi_i-\bm \phi_{i,h},\bm \phi_i-\bm \phi_{i,h}), \ 1\leq j\leq m,\label{eig_vec}\\
|\mu-\mu_{j,h}|&=O(h^{2(s-1)}), \ 1\leq j\leq m.\label{eigconv-01}
\end{align}
\end{theorem}
\begin{proof}
Note that $A$ and  $A_h$ are self-adjoint. We have that
\begin{align*}
((A-A_h)\bm \phi_i,\bm \phi_j)&=(\nabla\times\nabla\times(A-A_h)\bm\phi_i,\nabla\times\nabla\times A\bm\phi_j)\\
&=(\nabla\times\nabla\times(A-A_h)\bm\phi_i,\nabla\times\nabla\times (A-A_h)\bm\phi_j).
\end{align*}
Due to \cite[Thm 2.52]{Monk2003}, it holds that
\begin{align}\label{eigconv-02}
|\mu-\mu_{j,h}|\leq C\left\{\max_i\|\nabla\times\nabla\times(A-A_h)\bm\phi_i\|^2+\|(A-A_h)|_{E(\mu)}\|^2\right\}.
\end{align}
Let $\bm \phi \in E(\mu)$, 
\[
\|(A-A_h)\bm \phi\|_{H(\tc^2;\Omega)}\leq \inf_{\bm v_h\in X_h}\|A\bm \phi-\bm v_h\|_{H(\tc^2;\Omega)}
\leq Ch^{s-1}\left\|A\bm \phi\right\|_{s+1}
\leq {C{\mu}h^{s-1}}\left\|\bm \phi\right\|_{s+1}.
\]
In addition, we have that
\begin{align*}
\|(A-A_h)\bm \phi\|_{H(\tc^2;\Omega)}&\leq \inf_{\bm v_h\in X_h}\|A\bm \phi-\bm v_h\|_{H(\tc^2;\Omega)}=\mu\inf_{\bm v_h\in X_h}\|\bm \phi-({1}/{\mu})\bm v_h\|_{H(\tc^2;\Omega)}\\
&\leq \mu\|\bm \phi-\bm \phi_h\|_{H(\tc^2;\Omega)}{\color{blue}\leq}\mu \sqrt{a(\bm \phi-\bm \phi_h,\bm \phi-\bm \phi_h)}.
\end{align*}
Since $E(\mu)$ is finite dimensional, we obtain \eqref{eig_vec} and
\begin{align*}
\|(A-A_h)|_{E(\mu)}\|_{H(\tc^2)}&\leq {C_{\mu}h^{s-1}},
\end{align*}
which proves \eqref{eigconv-01}.
\end{proof}

\section{A posteriori error estimates for the eigenvalue problem}
Assume that $(\lambda, \bm u)\in \mathbb R\times H_0(\tc^2;\Omega)$ is a simple eigenpair of \eqref{eig-02} with
$\|\bm u\|_0=1$ and $(\lambda_h, \bm u_h)\in \mathbb R\times V_h^0$ is the associated finite element eigenpair of \eqref{eig-dis-01}
with $\|\bm u_h\|_0=1$. According to Theorem \ref{eigenvalue-convergence} and \cite[(3.28a)]{babuvska1989finite}, the following inequalities hold:
\begin{align}
&\|\bm u-\bm u_{h}\|\leq C\rho_\Omega(h)\3bar\bm u-\bm u_{h}\3bar,\label{error-eig-01}\\
&|\lambda_{h}-\lambda|\leq C \3bar\bm u-\bm u_{h}\3bar^2,\label{error-eig-02}
\end{align}
where $\3bar\bm u\3bar^2=a(\bm u,\bm u)$ and
\[
	\rho_{\Omega}(h)= \sup_{\bm f\in \bm L^2(\Omega),\|\bm f\|=1} \inf _{\bm v \in V^{0}_{h}}\left\|A\bm f-\bm v\right\|_{H(\text{curl}^2;\Omega)}.
\]
It is obvious that $\rho_{\Omega}(h) \rightarrow 0$ as $h\rightarrow 0$.

Define two projection operators $R_h, Q_h$ as follows.
For $\bm u\in H_0(\tc^2;\Omega)$ and $p\in H_0^1(\Omega)$, find $R_h \bm u\in V_h^0, Q_h p\in S_h^0$, such that
\begin{align*}
a(\bm u-R_h\bm u,\bm v_h)+b(\bm v_h,p-Q_h p)&=0,\quad \forall v_h\in V_h^0,\\
b(\bm u-R_h\bm u,q_h)&=0,\quad\forall q_h\in S_h^0.
\end{align*}
According to the orthogonality and the uniqueness of the discrete eigenvalue problem,
\begin{align*}
&\bm u_h=(\lambda_h+1)R_hA\bm u_h.
\end{align*}
Let $\left(\bm \omega^h; p^h\right)$ be the solution of \eqref{prob22} with $\bm f=(\lambda_h+1)\bm u_h$.
Then  
\begin{align}\label{omegah}
&\bm \omega^h=(\lambda_h+1)A\bm u_h
\quad \text{and}\quad \bm u_h=R_h \bm \omega^h. 
\end{align}
The following theorem relates the eigenvalue problem to a source problem with $\bm f=(\lambda_h+1)\bm u_h$.
\begin{theorem}\label{lemm7}
Let $r(h)=\rho_\Omega(h)+\3bar\bm u-\bm  u_h\3bar$. It holds that
\begin{align}\label{the-04}
\3bar\bm \omega^h-R_h\bm \omega^h\3bar-Cr(h)\3bar\bm u-\bm u_h\3bar\leq\3bar\bm u-\bm u_h\3bar\leq \3bar\bm \omega^h-R_h\bm \omega^h\3bar+Cr(h)\3bar\bm u-\bm u_h\3bar.
\end{align}
Furthermore, for $h$ small enough, there exist two constants $c$ and $C$ such that
\begin{align}\label{the-05}
c\3bar\bm \omega^h-R_h\bm \omega^h\3bar\leq\3bar\bm u-\bm u_h\3bar\leq C \3bar\bm \omega^h-R_h\bm \omega^h\3bar.
\end{align}
\end{theorem}
\begin{proof}
Since $\bm u_h=R_h \bm \omega^h$, by the triangle inequality, we have that
\begin{align*}
-\3bar\bm u-\bm \omega^h\3bar+\3bar\bm \omega^h-R_h\bm \omega^h\3bar\leq \3bar\bm u-\bm u_h\3bar&\leq\3bar\bm u-\bm \omega^h\3bar+\3bar\bm \omega^h-R_h\bm \omega^h\3bar.
\end{align*}
Using $\bm u=(\lambda+1)A\bm u$ and \eqref{omegah}, we obtain that
\begin{align}
\3bar\bm u-\bm \omega^h\3bar&=\3bar(\lambda+1) A\bm u-(\lambda_h+1)A\bm u_h\3bar\nonumber\\
&\leq|\lambda+1|\3bar A(\bm u-\bm u_h)\3bar+|\lambda -\lambda_h|\3barA\bm u_h\3bar.\label{the-01}
\end{align}
Due to the well-posedness of \eqref{prob22}, it holds that
\begin{align*}
&\3barA(\bm u-\bm u_h)\3bar\leq C\|\bm u-\bm u_h\|, 
\end{align*}
which, together with \eqref{error-eig-01} and \eqref{error-eig-02}, leads to
\begin{align}\label{the-03}
\3bar\bm u-\bm \omega^h\3bar\leq C r(h)\3bar\bm u-\bm u_h\3bar.
\end{align}
Then \eqref{the-04} follows immediately. Note that $r(h) \to 0$ as $h \to 0$. For $h$ small enough, \eqref{the-04} implies \eqref{the-05}. \end{proof}


We first derive an a posteriori error estimate when (a) $\bm f\in
H(\text{div}^0,\Omega)$ or (b) $\bm f$ is a vector polynomial for which $(\bm f,\nabla q_h)=0$, $\forall q_h\in S_h^0$.
Note that  $p=p_h=0$ for (a) and $p_h=0$ for (b). Hence $p_h=0$ holds for both cases.
%
%

Denote the total errors by $\bm e:=\bm u-\bm u_h \text{ and } \varepsilon:= p-p_h=p$.
Then $\bm e \in H_0(\tc^2;\Omega) \text{ and } \varepsilon\in  H_0^1(\Omega)$ satisfy the defect equations
\begin{align}\label{error-equation-1}
a( \bm e,\bm v)+b(\bm v, \varepsilon)&=r_1(\bm v), \ \forall \bm v\in H_0(\tc^2;\Omega),\\
\label{error-equation-2}
b(\bm e,q)&=r_2(\nabla q),\ \forall q\in H_0^1(\Omega),
\end{align}
where 
\[
r_1(\bm v)=(\bm f,\bm v)-\left((\nabla\times)^2\bm u_h,(\nabla\times)^2\bm v\right)-(\bm u_h,\bm v)-(\nabla p_h,\bm v)=(\bm f,\bm v)-\left((\nabla\times)^2\bm u_h,(\nabla\times)^2\bm v\right)-(\bm u_h,\bm v)
\] 
and $r_2(\nabla q)=-(\bm u_h,\nabla q)$.
We have the following Galerkin orthogonality
\begin{align}\label{orth-01}
r_1(\bm v_h)&=0, \ \forall \bm v_h\in V_h^0,\\
\label{orth-02}
r_2(\nabla q_h)&=0,\ \forall q_h\in S_h^0.
\end{align}

The error estimator will be constructed by employing
Lemma \ref{Helm}. Writing $\bm e=\bm e^0+\bm e^\bot$ and $\bm v = \bm v^0+\bm v^\bot$ with $\bm e^0, \bm v^0 \in \nabla H_0^1(\Omega)$ and $\bm e^\bot, \bm v^\bot \in X$, we obtain that
\begin{align}
\left(\bm e^0, \bm{v}^{0}\right)+\left(\bm{v}^{0}, \nabla\varepsilon\right) &=r_1(\bm{v}^0), \quad \forall \bm{v}^{0} \in \nabla H_0^1(\Omega), \label{irrotational}\\
\left((\nabla\times)^2\bm e^{\perp}, (\nabla\times)^2\bm v^{\perp}\right)+(\bm e^{\perp}, \bm v^{\perp}) &=r_1(\bm{v}^{\perp}), \quad \forall \bm{v}^{\perp} \in X,\label{irrotationa2}\\
(\bm e^0,\nabla q)&=r_2(\nabla q),\quad\forall q\in H_0^1(\Omega).\label{irrotationa3}
\end{align}
The estimators for the irrotational part $\bm e^0$, the solenoidal part $\bm e^\bot$, and $\nabla\varepsilon$ will be derived separately.
Firstly, consider the irrotational part $\bm e^0$ and $\nabla \varepsilon$. For a $\vartheta\in H_0^1(\Omega)$, we have
\[
r_1(\nabla \vartheta) =\sum_{T \in \mathcal{T}_{h}}\left(\bm f-\bm u_h, \nabla \vartheta\right)_T
=\sum_{T \in \mathcal{T}_{h}}-\left(\nabla\cdot ({\bm{f}-\bm u_h}),\vartheta\right)_T+\sum_{E \in \mathcal{E}_{h}^{\text{int}}}\left(\left[\![\bm{n}_E\cdot({\bm f}-\bm u_h)]\!\right]_{E}, \vartheta\right)_{E},
\]
where the jump
 \[\left[\![\bm{n}_E\cdot{\bm{u}}_{h}]\!\right]_{E}=\left(\bm{n}_E\cdot{\bm{u}}_{h}\right)_{E\subset T_2}-\left(\bm{n}_E\cdot{\bm{u}}_{h}\right)_{E\subset T_1},\]
  with $E\in \mathcal{E}_h^{\text{int}}$ the common edge of two adjacent elements $T_1, T_2\in \mathcal{T}_{h}$ and $\bm n_E$  the  unit normal vector of $E$ directed towards the interior of $T_1$.
We also have
\[
 r_2(\nabla \vartheta)=\sum_{T \in \mathcal{T}_{h}}-(\bm u_h, \nabla \vartheta)_T
 =\sum_{T \in \mathcal{T}_{h}}\left(\nabla \cdot \bm u_h, \vartheta\right)_T-\sum_{E \in \mathcal{E}_{h}^{\text{int}}}\left([\![\bm n_E\cdot \bm u_h]\!]_E,\vartheta\right)_E.
\]
We introduce the error terms which are related to  the upper and lower bounds for $\bm e^0$ and $\nabla\varepsilon$:
\begin{align}\label{errore0}
\eta_{0}:=\bigg(\sum_{T \in \mathcal{T}_{h}}\left(\eta_{0}^{T}\right)^{2}\bigg)^{1 / 2}+\bigg(\sum_{E \in \mathcal{E}_{h}^{\text{int}}}\left(\eta_{0}^{E}\right)^{2}\bigg)^{1 / 2},\\
\label{errore3}
\eta_{3}:=\bigg(\sum_{T \in \mathcal{T}_{h}}\left(\eta_{3}^{T}\right)^{2}\bigg)^{1 / 2}+\bigg(\sum_{E \in \mathcal{E}_{h}^{\text{int}}}\left(\eta_{3}^{E}\right)^{2}\bigg)^{1 / 2},
\end{align}
where
\begin{align*}
\eta_{0}^{T} &:= h_{T}\left\|\nabla\cdot ({\bm f}-\bm u_h)\right\|_T ,\  T \in \mathcal{T}_{h}, \\
\eta_{0}^{E} &:=h_{E}^{1 / 2}\left\|[\![\bm n_E\cdot (\bm f-\bm u_h)]\!]_E\right\|_{E} ,\ E \in \mathcal{E}_{h}^{\text{int}},\\
\eta_{3}^{T} &:= h_{T}\left\|\nabla\cdot \bm u_h\right\|_T ,\  T \in \mathcal{T}_{h}, \\
\eta_{3}^{E} &:=h_{E}^{1 / 2}\left\|[\![\bm n_E\cdot \bm u_h]\!]_E\right\|_{E} ,\ E \in \mathcal{E}_{h}^{\text{int}}.
\end{align*}
Next, we consider the bounds for $\bm e^\bot$. For $\bm w \in X$, the residual $r_1(\bm w)$ can be expressed as
\begin{eqnarray*}
 r_1(\bm w)&=&\sum_{T \in \mathcal{T}_{h}}\big(\bm f-\bm u_h, \bm{w}\big)_T-\left((\nabla\times)^2{\bm{u}}_{h}, (\nabla\times)^2\bm{w}\right)_T \\
 &=&\sum_{T \in \mathcal{T}_{h}}\left(\bm f-(\nabla\times)^4{\bm{u}}_{h}-\bm u_h, \bm{w}\right)_T-\sum_{E \in \mathcal{E}_{h}^{\text{int}}}\left([\![ (\nabla\times)^2 {\bm{u}}_{h} \times\bm{n}_E]\!]_{E}, \nabla\times\bm{w}\right)_{E}\\
&& ~ -\sum_{E \in \mathcal{E}_{h}^{\text{int}}}\left([\![(\nabla\times)^3 {\bm{u}}_{h}]\!]_{E}, \bm n_E\times\bm{w}\right)_{E},
\end{eqnarray*}
where $[\![(\nabla\times)^2\bm u_h\times\bm n_E]\!]_E$ stands for the jump of the tangential component of $(\nabla\times)^2\bm u_h$ and $[\![(\nabla\times)^3 {\bm{u}}_{h}]\!]_{E}$ stands for the jump of $(\nabla\times)^3 {\bm{u}}_{h}$.
The bounds for $\3bar\bm e^\bot \3bar$ contain the error terms
\begin{align}
\label{errore1}\eta_{1} &:=\bigg(\sum_{T \in \mathcal{T}_{h}}\left(\eta_{1}^{T}\right)^{2}\bigg)^{1 / 2}+\bigg(\sum_{E \in \mathcal{E}_{h}^{\text{int}}}\left(\eta_{1;1}^{E}\right)^{2}\bigg)^{1 / 2}+\bigg(\sum_{E \in \mathcal{E}_{h}^{\text{int}}}\left(\eta_{1;2}^{E}\right)^{2}\bigg)^{1 / 2},\\
\label{errore2}\eta_{2} &:=\bigg(\sum_{T \in \mathcal{T}_{h}}\left(\eta_{2}^{T}\right)^{2}\bigg)^{1 / 2},
 \end{align}
where
\begin{align*}
\eta_{1}^{T} &:= h_{T}^2\left\|\pi_{h} \bm{f}-(\nabla\times)^4{\bm{u}}_{h}-\bm u_h\right\|_T, \quad T \in \mathcal{T}_{h}, \\
\eta_{2}^{T} &:= h_{T}^2\left\|\bm{f}-\pi_{h} \bm{f}\right\|_T, \quad T \in \mathcal{T}_{h}, \\
\eta_{1;1}^{E} &:= h_{E}^{1/ 2}\left\|[\![\bm{n}_E \times (\nabla\times)^2 {\bm{u}}_{h}]\!]_{E}\right\|_{E},
 \quad E \in \mathcal{E}_{h}^{\text {int }},\\
\eta_{1;2}^{E} &:= h_{E}^{3 / 2}\left\|[\![ (\nabla\times)^3 {\bm{u}}_{h}]\!]_{E}\right\|_{E}, \quad E \in \mathcal{E}_{h}^{\text {int }},
 \end{align*}
 and $\pi_{h} \bm{f}$ denotes the $\bm L^2$-projection of $\bm f$ onto $\bm P_k(T)$.

 Now we state the a posteriori estimate for $\bm e$ and $\varepsilon$ in the energy norm.
 \begin{theorem}\label{posteriori-estimate-main}
 Let $\eta_0, \ \eta_1$, $\eta_2$, and $\eta_3$ be defined in \eqref{errore0}, \eqref{errore1}, \eqref{errore2}, and \eqref{errore3}, respectively.
  Then, if $h< 1$, 
 \begin{align*}
 \gamma_1(\eta_{0}+\eta_{1}+\eta_3)-\gamma_2\eta_2\leq\3bar\bm e\3bar+\|\nabla \varepsilon\|\leq\Gamma_1(\eta_{0}+\eta_{1}+\eta_3)+\Gamma_2\eta_2
  \end{align*}
 and, if $h$ is small enough,
  \begin{align*}
 \gamma_3(\eta_{1}+\eta_3)-\gamma_4(\eta_2+h^2\eta_{0})\leq\3bar\bm e\3bar\leq\Gamma_3(\eta_{0}+\eta_{1}+\eta_3)+\Gamma_4\eta_2,
  \end{align*}
  where $\gamma_1, \gamma_2, \gamma_3, \gamma_4, \Gamma_1, \Gamma_2, \Gamma_3$, and $\Gamma_4$ are some constants independent of $h$.
 \end{theorem}

\begin{proof}
Since $\bm e = \bm e^0+\bm e^\bot$, the proof is split into two parts.

{\bf (i) Estimation of the irrotational part of the error.}
Based on \eqref{irrotationa3}, we have the following uniformly positive definite variational problem on $H^1_0(\Omega)$. Seek $\varphi\in H_0^1(\Omega)$ such that
\begin{align}
(\nabla\varphi,\nabla q)=r_2(\nabla q),\quad\forall q\in H^1_0(\Omega).
\end{align}
Note that $r_2(\nabla q_h)=0, \forall q_h \in S_h^0$ and $\bm e^0=\nabla \varphi$ for some $\varphi$.
Define a projection operator $P_h^k: H^1_0(\Omega)\longrightarrow S_h^0$ such that (see, e.g., \cite{beck2000residual,osborn1975spectral, scott1990finite})
\begin{align}
&P_h^k \phi=\phi,\quad \forall \phi\in S_h^0,\label{app-01}\\
&\|\phi-P_h^k \phi\|_T\leq Ch_T\|\nabla \phi\|_{\omega_T}, \label{app-02}\\
&\|\phi-P_h^k \phi\|_{L^2(E)}\leq C\sqrt{h_E}\|\nabla \phi\|_{\omega_E},\label{app-03}\\
&\|\nabla P_h^k \phi\|_T\leq C\|\nabla\phi\|_{\omega_T}.\label{app-04}
\end{align}

Due to \eqref{irrotationa3} and the orthogonal property \eqref{orth-02}, we have that
\begin{align}
\|\bm e^0\|^2=r_2(\bm e^0)=r_2(\nabla\varphi-\nabla P_h^k\varphi)=-(\bm u_h,\nabla(\varphi-P_h^k\varphi)).
\end{align}
Using integration by parts, \eqref{app-02}, and \eqref{app-03},  we obtain that
\begin{eqnarray*}
\-\left(\bm u_h,\nabla(\varphi-P_h^k\varphi)\right)
&=&\sum_{T\in \mathcal{T}_h}-(\nabla\cdot\bm u_h,\varphi-P_h^k\varphi)_T+\sum_{E\in \mathcal{E}_h^{\text{int}}}([\![\bm n\cdot \bm u_h]\!]_E,\varphi-P_h^k\varphi)_E\\
&\le &C \sum_{T\in \mathcal{T}_h}\|\nabla\cdot\bm u_h\|_Th_T\|\nabla \varphi\|_{\omega_T}+C\sum_{E\in \mathcal{E}_h^{\text{int}}}\|[\![\bm n\cdot \bm u_h]\!]_E\|_{L^2(E)}\sqrt{h_E}\|\nabla \varphi\|_{\omega_E}\\
&\le & C \Big(\sum_{T\in \mathcal{T}_h}\|\nabla\cdot\bm u_h\|_T^2h_T^2\Big)^{1/2}\|\bm e^0\|+
C \Big(\sum_{E\in \mathcal{E}_h^{\text{int}}}\|[\![\bm n\cdot \bm u_h]\!]_E\|_{L^2(E)}^2{h_E}\Big)^{1/2}\|\bm e^0\|.
\end{eqnarray*}
Therefore, we have
\begin{align}\label{irrota-01}
\|\bm e^0\|
&\leq C\eta_3.
\end{align}

Similarly, we can obtain the upper bounds of  $\|\nabla\varepsilon\|$. Due to \eqref{irrotational} and \eqref{irrotationa3} , we have
\begin{align}
\|\nabla\varepsilon\|^2=r_1(\nabla\varepsilon)-r_2(\nabla \varepsilon)=r_1 (\nabla \varepsilon-\nabla P^k_h\varepsilon)- r_2( \nabla \varepsilon-\nabla P^k_h\varepsilon).
\end{align}
By Green's formula, \eqref{app-02}, and \eqref{app-03},
\begin{align*}
\|\nabla\varepsilon\|^2
&=\sum_{T\in \mathcal{T}_h}-\left(\nabla\cdot\left(\bm f-\bm u_h\right),\varepsilon-P_h^k\varepsilon\right)-\sum_{T\in \mathcal{T}_h}\left(\nabla\cdot\bm u_h,\varepsilon-P_h^k\varepsilon\right)\\
& \quad +\sum_{E\in \mathcal{E}_h^{\text{int}}}\left([\![\bm n\cdot\left(\bm f-\bm u_h)\right]\!]_E,\varepsilon-P_h^k\varepsilon\right)_E
+\sum_{E\in \mathcal{E}_h^{\text{int}}}\left([\![\bm n\cdot\bm u_h]\!]_E,\varepsilon-P_h^k\varepsilon\right)_E\\
&\leq \sum_{T\in \mathcal{T}_h}\|\nabla\cdot(\bm f-\bm u_h)\|_Th_T\|\nabla \varepsilon\|_{\omega_T}+\|[\![\bm n\cdot (\bm f-\bm u_h)]\!]_E\|_{L^2(E)}\sqrt{h_E}\|\nabla\varepsilon\|_{\omega_E}\\
& \quad +\sum_{T\in \mathcal{T}_h}\|\nabla\cdot\bm u_h\|_Th_T\|\nabla \varepsilon\|_{\omega_T}+\|[\![\bm n\cdot \bm u_h]\!]_E\|_{L^2(E)}\sqrt{h_E}\|\nabla\varepsilon\|_{\omega_E}\\
&\le  C\Big(\sum_{T\in \mathcal{T}_h}\|\nabla\cdot(\bm f-\bm u_h)\|_T^2h_T^2\Big)^{1/2}\|\nabla \varepsilon\|+
C \Big(\sum_{T\in \mathcal{T}_h}\|\nabla\cdot\bm u_h\|_T^2h_T^2\Big)^{1/2}
\|\nabla \varepsilon\|\\
& \quad + C\Big(\sum_{E\in \mathcal{E}_h^{\text{int}}}\|[\![\bm n\cdot (\bm f-\bm u_h)]\!]_E\|_{L^2(E)}^2{h_E}\Big)^{1/2}\|\nabla \varepsilon\|+
C\Big(\sum_{E\in \mathcal{E}_h^{\text{int}}}\|[\![\bm n\cdot \bm u_h]\!]_E\|_{L^2(E)}^2{h_E}\Big)^{1/2}\|\nabla \varepsilon\|.
\end{align*}
Therefore, we have that
\begin{align}\label{irrota-varepsilon}
\|\nabla \varepsilon\|
\leq C\left(\eta_0+\eta_3\right).
\end{align}

We now derive the lower bounds of $\bm e^0$ and $\|\nabla \varepsilon\|$ using the bubble functions.
Denote by $\lambda_1^T,\lambda_2^T,\lambda_3^T$ the barycentric coordinates of $T\in\mathcal{T}_h$ and define the bubble function $b_T$ by
\begin{align*}
b_T=
\left\{
  \begin{array}{ll}
    27\lambda_1^T\lambda_2^T\lambda_3^T, & {\text{on}\ T,} \\
    0, & {\Omega \setminus T.}
  \end{array}
\right.
\end{align*}
Given $E\in \mathcal{E}_h$, a common edge of $T_1$ and $T_2$,
let $\omega_E=T_1\cup T_2$ and enumerate the vertices of $T_1$ and $T_2$ such that the vertices of $E$ are numbered first.
Define the edge-bubble function $b_E$ by
\begin{align*}
b_E=
\left\{
  \begin{array}{ll}
    4\lambda_1^{T_i}\lambda_2^{T_i}, & {\text{on}\ T_i,\ i=1, 2,} \\
    0, & {\Omega \setminus \omega_E.}
  \end{array}
\right.
\end{align*}
Using the technique in \cite{ainsworth1997posteriori}, we have the following norm equivalences.
\begin{align}
&\|b_T\phi_h\|_T\leq\|\phi_h\|_T\leq C\|b_T^{1/2}\phi_h\|_T, \forall \phi_h\in P_k(T),\label{eq-nor-01}\\
&\|b_E\phi_h\|_E\leq\|\phi_h\|_E\leq C\|b_E^{1/2}\phi_h\|_E, \forall \phi_h\in P_k(E),\label{eq-nor-02}\\
&{\|b_E\phi_h\|_T\leq\|\phi_h\|_T, \forall \phi_h\in P_k(T)}.\label{eq-nor-03}
\end{align}
Using \eqref{eq-nor-01}, integration by parts, the inverse inequality, and the fact that  $b_T\nabla\cdot\bm u_h\in H_0^1(T)\subset H_0^1(\Omega)$,
we have that
\begin{align*}
\frac{(\eta_3^T)^2}{h_T^2}&=\|\nabla\cdot\bm u_h\|_T^2\leq C\|b_T^{1/2}\nabla\cdot\bm u_h\|_T^2=C\int_Tb_T(\nabla\cdot\bm u_h)^2\d \bm x\\
&=-C\int_T\bm u_h\nabla(b_T\nabla\cdot\bm u_h)\d \bm x=Cr_2\big(\nabla(b_T\nabla\cdot\bm u_h)\big)\\
&=C\left(\bm e^0,\nabla(b_T\nabla\cdot\bm u_h)\right)\leq C\|\bm e^0\|_T\|\nabla(b_T\nabla\cdot\bm u_h)\|_T\\
&\leq \frac{C}{h_T}\|\bm e^0\|_T\|\nabla\cdot\bm u_h\|_T,
\end{align*}
which implies that
\begin{align}\label{eta3T}
\eta_3^T\leq C\|\bm e^0\|_T.
\end{align}
Extend $[\![\bm n\cdot \bm u_h]\!]_E$ to $[\![\bm n\cdot \bm u_h]\!]_{E;T_i}$ defined on $T_i$ such that
\begin{align}
	\left\|[\![\bm n\cdot \bm u_h]\!]_{E;T_i}\right\|_{T_i}\leq Ch^{1/2}_{T_i}\|[\![\bm n\cdot \bm u_h]\!]_E\|_E.
\end{align}
The estimate of the local upper bound for $\eta_3^E$ can be obtained similarly:
\begin{eqnarray*}
\frac{(\eta_3^E)^2}{h_E}&=&\|[\![\bm n\cdot \bm u_h]\!]_E\|_E^2\leq C\int_E[\![\bm n\cdot \bm u_h]\!]_E^2b_E\d s\\
&=&C\left(\sum_{i=1}^2\int_{T_i}\bm u_h\nabla(b_E[\![\bm n\cdot \bm u_h]\!]_{E;T_i})+\nabla\cdot\bm u_hb_E[\![\bm n\cdot \bm u_h]\!]_{E;T_i}\d  \bm x\right)\\
&=& -C r_2\big(\nabla(b_E[\![\bm n\cdot \bm u_h]\!]_{E;T_1\cup T_2})\big)+C\int_{T_1\cup T_2}\nabla\cdot\bm u_hb_E[\![\bm n\cdot \bm u_h]\!]_{E;T_i}\d  \bm x\\
&\leq & C\sum_{i=1}^2\left(h_{T_i}^{-1}\|\bm e^0\|_{T_i}+\|\nabla\cdot\bm u_h\|_{T_i}\right){\eta_3^E},
\end{eqnarray*}
where we have used the fact that
\[
\|\nabla(b_E[\![\bm n \cdot \bm u_h]\!]_{E;T_i})\|_{T_i}\leq Ch_{T_i}^{-1}\|b_E[\![\bm n \cdot \bm u_h]\!]_{E;T_i}\|_{T_i}\leq Ch_{T_i}^{-1/2}\|[\![\bm n \cdot \bm u_h]\!]_E\|_{E}.
\]
Consequently,
\begin{align}\label{eta3E}
\eta_3^E\leq C\big(\|\bm e^0\|_{\omega_E}+\eta_3^{T_1}+\eta_3^{T_2}\big)\leq C\|\bm e^0\|_{\omega_E}.
\end{align}
Now collecting \eqref{irrota-01}, \eqref{eta3T}, and \eqref{eta3E}, we have that
\begin{align}\label{est-eta3}
	c\eta_3\leq\|\bm e^0\|\leq C\eta_3.
\end{align}
Similarly,
\begin{align}\label{eta0T}
\eta_0^T\leq C\left(\|\nabla\varepsilon\|_T+\|\bm e^0\|_T\right),
\end{align}
\begin{align}\label{eta0E}
\eta_0^E\leq C\left(\|\nabla\varepsilon\|_{\omega_T}+\|\bm e^0\|_{\omega_T}\right)+\eta_0^{T_1}+\eta_0^{T_2}\leq C\left(\|\nabla\varepsilon\|_{\omega_T}+\|\bm e^0\|_{\omega_T}\right).
\end{align}
Combining \eqref{irrota-01}, \eqref{irrota-varepsilon}, \eqref{eta0T}, and \eqref{eta0E}, we obtain that
\begin{align}\label{este0}
c(\eta_0+\eta_3)\leq\|\nabla \varepsilon\|+\|\bm e^0\|\leq C(\eta_0+\eta_3).
\end{align}

{\bf (ii) Estimation of the solenoidal part $\bm e^\bot$.}
We  start with proving the upper bound for $\eta_{1}^T$ by using $b_T$ again.
Employing the similar technique in \cite{ainsworth1997posteriori}, we have the following estimates for any $\bm v$ in finite dimensional spaces:
\begin{align}
	&\|\bm v\|_{T}^{2} \leq C \int_{T} b^2_{T} \bm v^{2}\d\bm x,\label{est1}\\
	&\left\|b^2_{T} \bm v\right\|_{T} \leq\|\bm v\|_{T}.\label{est2}
\end{align}
Setting $\bm \phi_h=\pi_{h} \bm{f}-(\nabla\times)^4{\bm{u}}_{h}-\bm u_h$, we have that
\begin{align*}
\left(\frac{\eta_{1}^T}{h_T^2}\right)^2&=\left\|\pi_{h} \bm{f}-(\nabla\times)^4{\bm{u}}_{h}-\bm u_h\right\|_T^2\\
&\leq C\int_T(\bm f-(\nabla\times)^4{\bm{u}}_{h}-{\bm{u}}_{h})b_T^2\bm \phi_h+(\pi_{h} \bm{f}-\bm f)b_T^2\bm \phi_h\d \bm x \quad \big(\text{By \eqref{est1}}\big)\\
&=C\Big(r_1( b_T^2\bm \phi_h)+\int_T(\pi_{h} \bm{f}-\bm f)b_T^2\bm \phi_h\d \bm x \Big)\quad \Big( b_T^2\bm \phi_h\in H_0(\text{curl}^2;\Omega)\Big)\\
&=C\Big(a(\bm e, b_T^2\bm \phi_h)+b(b_T^2\bm \phi_h,\varepsilon)+\int_T(\pi_{h} \bm{f}-\bm f)b_T^2\bm \phi_h\d \bm x \Big)\quad \big( \text{By } \eqref{error-equation-1}\big)\\
&\leq C \threeBar\bm e\threeBar_{T}\threeBar b_T^2\bm \phi_h\threeBar_{T}+C\big\|\nabla\varepsilon\big\|_T\left\|b_T^2\bm \phi_h\right\|_T+C\eta_{2}^Th_T^{-2}\left\|b_T^2\bm \phi_h\right\|_T.
\end{align*}
Due to the inverse inequality and \eqref{est2}, it holds that
\begin{align*}
\3barb_T^2\bm \phi_h\3bar_{T}&=\|b_T^2\bm \phi_h\|_T+\|(\nabla\times)^2b_T^2\bm \phi_h\|_T\\
&\leq\|b_T^2\bm \phi_h\|_T+C h_T^{-1}\|\nabla\times b_T^2\bm \phi_h\|_{T}\\
&\leq C h_T^{-2}\|b_T^2\bm \phi_h\|_{T}\leq C h_T^{-2}\bm \|\bm\phi_h\|_{T}.
\end{align*}
Thus we obtain that 
\begin{align*}
\left(\frac{\eta_{1}^T}{h_T^2}\right)^2
&\leq C\left( h_T^{-2}\3bar\bm e\3bar_{T}+\|\nabla\varepsilon\|_T+h_T^{-2}\eta_{2}^T\right)\|\bm \phi_h\|_T.
\end{align*}
Dividing the above inequality by $\left\|\bm \phi_h\right\|_T$ and multiplying by $h_T^2$, we obtain
\begin{align}\label{eta1T-upper}
	\eta_{1}^T\leq C\left(\3bar\bm e\3bar_{T}+h_T^2\|\nabla\varepsilon\|_T+\eta_{2}^T\right).
\end{align}

Next we  estimate the upper bound for $\eta_{1;1}^E$ by using the bubble functions $b_T, b_E$. Let $T_1$ and $T_2$ be two elements sharing the edge $E$.
We extend the jump $[\![\bm n\times(\nabla\times)^2\bm u_h]\!]_E$ defined on $E$ to two polynomial functions $[\![\bm n\times(\nabla\times)^2\bm u_h]\!]_{E;T_1}$
defined on $T_1$ and $[\![\bm n\times(\nabla\times)^2\bm u_h]\!]_{E;T_2}$ defined on $T_2$ such that, for $1\leq i\leq 2$,
\begin{align}\label{extension-estimate}
\|[\![\bm n\times(\nabla\times)^2\bm u_h]\!]_{E;T_i}\|_{T_i}\leq C h_{T_i}^{1/2}\|[\![\bm n\times(\nabla\times)^2\bm u_h]\!]_{E}\|_{E}.
\end{align}
Denote $ \psi_h|_{T_i}=[\![\bm n\times(\nabla\times)^2\bm u_h]\!]_{E;T_i}$ for $i=1, 2$ and $\bm \omega_{E,1}=(b_{T_1}-b_{T_2})b_E\psi_h\bm \tau_E$.
A simple calculation shows that
\begin{align*}
(\nabla\times\bm\omega_{E,1})|_E=\frac{27}{8}\left(\frac{h_E}{|T_1|}+\frac{h_E}{|T_2|}\right)b_E^2\psi_h.
\end{align*}
Similar to \eqref{est1} and \eqref{est2}, the following inequalities hold
\begin{align}
	&\| v\|_{E} \leq C\| b_Ev\|_{E},\label{est3}\\
	&\left\|(b_{T_1}-b_{T_2})b_E v\right\|_{T_1\cup T_2} \leq\|v\|_{T_1\cup T_2}.\label{est4}
\end{align}
Now we are ready to construct the upper bound for $\eta_{1;1}^E$:
\begin{eqnarray*}
&&h_E^{-1}\|[\![\bm n\times(\nabla\times)^2\bm u_h]\!]_E\|_E^2\\
&\leq &C\int_E[\![\bm n\times(\nabla\times)^2\bm u_h]\!]_E \nabla \times \bm \omega_{E,1}\d s\quad(\text{By \eqref{est3}})\\
&=&C\int_{T_1\cup T_2}(\nabla\times)^4\bm u_h\cdot \bm \omega_{E,1}-(\nabla\times)^2\bm u_h\cdot (\nabla\times)^2\bm \omega_{E,1}\d \bm x\quad(\text{By in tegration by parts})\\
&=&C\left(r_1( \bm \omega_{E,1})-(\bm f -\bm u_h-(\nabla\times)^4\bm u_h,\bm \omega_{E,1})\right)\quad\Big(\bm \omega_{E,1}\in H_0(\text{curl}^2;\Omega)\Big)\\[1.5mm]
&\leq & C \3bar\bm e\3bar_{T_1\cup T_2}\3bar\bm \omega_{E,1}\3bar_{T_1\cup T_2}+C\|\bm \omega_{E,1}\|_{T_1\cup T_2}\|\nabla\varepsilon\|_{T_1\cup T_2}\\[1.5mm]
&& +C\left(h_{T_1}^{-2}(\eta_{1}^{T_1}+\eta_{2}^{T_1})+h_{T_2}^{-2}(\eta_{1}^{T_2}+\eta_{2}^{T_2})\right)
\|\bm \omega_{E,1}\|_{T_1\cup T_2}.
\end{eqnarray*}
By applying the inverse inequality, \eqref{extension-estimate}, and \eqref{est4}, we get
\[\!|\!|\bm \omega_{E,1}\3bar_{T_1\cup T_2}\leq h_E^{-2}\|\bm \omega_{E,1}\|_{T_1\cup T_2}\leq h_E^{-3/2}\|[\![\bm n\times(\nabla\times)^2\bm u_h]\!]_E\|_E,\]
 which, together with \eqref{eta1T-upper}, leads to
\begin{align}\label{etaE-1}
\eta_{1;1}^E\leq C\left(\3bar\bm e\3bar_{T_1\cup T_2}+\eta_{2}^{T_1}+\eta_{2}^{T_2}+h_{T_1}^2\|\nabla\varepsilon\|_{T_1}+h_{T_2}^2\|\nabla\varepsilon\|_{T_2}\right).
\end{align}

The upper bound for $\eta_{1;2}^E$ can be constructed in a similar way.
Extend $[\![(\nabla\times)^3\bm u_h]\!]_{E}$ to $[\![(\nabla\times)^3\bm u_h]\!]_{E;T_i}$ on $T_i$ such that
\begin{align}\label{extension-estimate2}
\|[\![(\nabla\times)^3\bm u_h]\!]_{E;T_i}\|_{T_i}\leq C h_{T_i}^{1/2}\|[\![(\nabla\times)^3\bm u_h]\!]\|_{E}.
\end{align}
Denote $\bm \omega_{E,2}|_{T_i}:=b_E^2[\![(\nabla\times)^3\bm u_h]\!]_{E;T_i}\bm \tau_E$ with $\bm \tau_E$ such that $\bm n_E\times\bm \tau_E=1$.
Then $\bm n_E\times \bm\omega_{E,2}|_{E}=b_E^2[\![(\nabla\times)^3 \bm u_h]\!]_{E}$. Hence,
\begin{eqnarray*}
&&\|[\![(\nabla\times)^3\bm u_h]\!]_E\|_E^2
\lesssim \int_E[\![(\nabla\times)^3\bm u_h]\!]_E\bm n_E\times \bm \omega_{E,2} \d s\\
&= &-\left((\nabla\times)^4\bm u_h, \bm \omega_{E,2}\right)_{T_1\cup T_2}+\left((\nabla\times)^2\bm u_h,(\nabla\times)^2\bm \omega_{E,2}\right)_{T_1\cup T_2}+\int_E [\![\bm n\times(\nabla\times)^2\bm u_h]\!]_E\nabla\times\bm \omega_{E,2} \d s\\
&\le & C\left(r_1( \bm \omega_{E,2})-\left(\bm f-\bm u_h -(\nabla\times)^4\bm u_h,\bm\omega_{E,2}\right)+\int_E [\![\bm n\times(\nabla\times)^2\bm u_h]\!]_E\nabla\times\bm \omega_{E,2} \d s\right)\\
&\le & C h_E^{-3/2}\left(\eta_{1}^{T_1}+\eta_{2}^{T_1}+\eta_{1}^{T_2}+\eta_{2}^{T_2}+\eta_{1;1}^E+\3bar\bm e\3bar_{T_1\cup T_2}+h_{T_1}^2\|\nabla\varepsilon\|_{T_1}+h_{T_2}^2\|\nabla\varepsilon\|_{T_2}\right)\left\|[\![(\nabla\times)^3\bm u_h]\!]_E\right\|_{E}.
\end{eqnarray*}
Dividing the above inequality by $\left\|[\![(\nabla\times)^3\bm u_h]\!]_E\right\|_{E}$ and applying \eqref{eta1T-upper} and \eqref{etaE-1}, we obtain
\begin{align}\label{eta_12_E}
\eta_{1,2}^E\le C\left( \eta_{2}^{T_1}+\eta_{2}^{T_2}+\3bar\bm e\3bar_{T_1\cup T_2}+h_{T_1}^2\|\nabla\varepsilon\|_{T_1}+h_{T_2}^2\|\nabla\varepsilon\|_{T_2}\right).
\end{align}
Collecting \eqref{eta1T-upper},\eqref{etaE-1}, and \eqref{eta_12_E}, we have that
\begin{align}\label{eta1-upper}
	\eta_1\le C\left( \eta_{2}+\3bar\bm e\3bar+h^2\|\nabla\varepsilon\| \right).
\end{align}

It remains to construct the upper bound of $\bm e^{\bot}$.
For $\bm e^\bot\in X\subset H_0(\tc^2;\Omega)$, according to Lemma \ref{Helm}, $\bm e^{\perp}=\bm w+\nabla \psi$ with $\bm w\in \bm H^2(\Omega)$ and $\psi\in H_0^1(\Omega)$,
we have
\begin{align*}
\3bar\bm e^\bot\3bar^2=r_1( \bm e^\bot)=r_1( \bm w)+r_1( \nabla \psi).
\end{align*}
Due to the Galerkin orthogonality \eqref{orth-01}, for any $\bm w_h\in V_h$,  
\begin{eqnarray*}
&&r_1( \bm w)=r_1( \bm w-\bm w_h)\\
&=&\sum_{T\in \mathcal{T}_h}\Big(\left(\bm f-\bm u_h-(\nabla\times)^4\bm u_h, \bm w-\bm w_h\right)-\sum_{E\in\mathcal  E_h(T)}\int_E\bm n\times(\nabla\times)^2\bm u_h\nabla\times(\bm w-\bm w_h)\d s\\
&&-\sum_{E\in \mathcal  E_h(T)}\int_E(\nabla\times)^3\bm u_h \bm n\times (\bm w-\bm w_h)\d s\Big)\\
&\leq & \sum_{T\in \mathcal{T}_h}\Big(\|\bm\pi _h\bm f-\bm u_h-(\nabla\times)^4\bm u_h\|_T\|\bm w-\bm w_h\|_T+\|\bm\pi_h\bm f-\bm f\|_T\|\bm w-\bm w_h\|_T\Big)\\
&&+\sum_{E\in \mathcal{E}_h^{\text{int}}}\Big(\|[\![\bm n \times (\nabla\times)^2\bm u_h]\!]_E\|_E\|\nabla\times(\bm w-\bm w_h)\|_E+\|[\![(\nabla\times)^3\bm u_h]\!]_E\|_E\|\bm n\times(\bm w-\bm w_h)\|_E\Big)\\
&\leq & C\bigg(\sum_{T\in \mathcal{T}_h}\Big(h_T^4\|\bm\pi _h\bm f-\bm u_h-(\nabla\times)^4\bm u_h\|_T^2+h_T^4\|\bm\pi_h \bm f-\bm f\|_T^2+\sum_{E\in\mathcal  E_h(T)} h_E^3\|[\![(\nabla\times)^3\bm u_h]\!]_E\|_E^2\Big)\\
&&\qquad+\sum_{E\in \mathcal  E_h(T)} h_E\|[\![\bm n \times (\nabla\times)^2\bm u_h]\!]_E\|_E^2\bigg)^{1/2}\\
&&\bigg(\sum_{T\in \mathcal{T}_h}\Big(h_T^{-4}\|\bm w-\bm w_h\|_T^2+\sum_{E\in \mathcal  E_h(T)} h_E^{-1}\|\nabla\times(\bm w-\bm w_h)\|_E^2+\sum_{E\in \mathcal  E_h(T)} h_E^{-3}\|\bm w-\bm w_h\|_E^2\Big)\bigg)^{1/2}.
\end{eqnarray*}
Let $\bm w_h=\Pi_{C}\bm w$. According to the  trace inequality and Theorem \ref{clmt-error}, we obtain
\begin{align*}
\sum_{T\in \mathcal{T}_h}\Big(h_T^{-4}\|\bm w-\bm w_h\|_T^2+\sum_{E\in \mathcal  E_h(T)} h_E^{-1}\|\nabla\times(\bm w-\bm w_h)\|_E^2+\sum_{E\in \mathcal  E_h(T)} h_E^{-3}\|\bm w-\bm w_h\|_E^2\Big)
\le C \|\bm w\|_2^2.
\end{align*}
Furthermore, we use \eqref{decom-01}, \eqref{decom-02}, and the Poincar\'e inequality to obtain
\begin{align*}
r_1( \bm w)&\leq C(\eta_1+\eta_2)\|\bm w\|_2\leq C(\eta_1+\eta_2)\|\nabla\times\bm e^\bot\|_1\leq C(\eta_1+\eta_2)\3bar\bm e^\bot\3bar.
\end{align*}
Similar to the proof of \eqref{irrota-varepsilon}, using  \eqref{decom-03}, it holds that
\begin{align*}
r_1( \nabla \psi)\leq C(\eta_0+\eta_3)\|\nabla \psi\|\leq C (\eta_0+\eta_3)\3bar\bm e^\bot\3bar.
\end{align*}
Hence,
\begin{align}\label{e-upper}
\3bar\bm e^\bot\3bar\leq C(\eta_0+\eta_1+\eta_2+\eta_3).
\end{align}
Combining \eqref{irrota-varepsilon}, \eqref{est-eta3}, \eqref{este0}, \eqref{eta1-upper}, and \eqref{e-upper}, we obtain Theorem \ref{posteriori-estimate-main}.
\end{proof}


When $\bm f=(\lambda_h+1)\bm u_h$, according to the definition of $\eta_0, \eta_2$, and $\eta_3$,  we have that $\eta_0=\lambda_h\eta_3$ and $\eta_2=0$.
The following error estimator is a direct consequence of Theorem \ref{posteriori-estimate-main} and \eqref{error-eig-02}.
\begin{theorem}
For $h$ small enough, there exist constants $c_1, C_1$, and $C_2$ such that
\begin{align*}
c_1(\eta_1+\eta_3)\leq\3bar\bm u-\bm u_h\3bar\leq C_1(\eta_1+(\lambda_h+1)\eta_3),
\end{align*}
and
\begin{align*}
|\lambda-\lambda_h|\leq C_2(\eta_1+(\lambda_h+1)\eta_3)^2,
\end{align*}
where $\eta_1$ and $\eta_3$ are respectively defined in \eqref{errore1} and \eqref{errore3} with $\bm f=(\lambda_h+1)\bm u_h$.
\end{theorem}

\section{Numerical Examples}
\subsection{A priori error estimate}
Consider three domains:
\begin{itemize}
\item $\Omega_1$: the unit square given by $(0, 1) \times (0, 1)$,
\item $\Omega_2$: the L-shaped domain given by $(0, 1) \times (0, 1) \slash [1/2, 1) \times (0, 1/2]$,
\item $\Omega_3$: given by $(0, 1) \times (0, 1) \slash [1/4, 3/4] \times [1/4, 3/4]$.
\end{itemize}
The initial meshes of the domains are shown in Figure \ref{fig1}. In Tables \ref{tab1}, \ref{tab3}, and \ref{tab5}, we list the first five eigenvalues.
Tables \ref{tab2}, \ref{tab4}, and \ref{tab6} show the convergence rates of the relative errors for the first eigenvalues, which agree with the theory.

\begin{figure}
	\centering
	\includegraphics[width=0.3\linewidth, height=0.2\textheight]{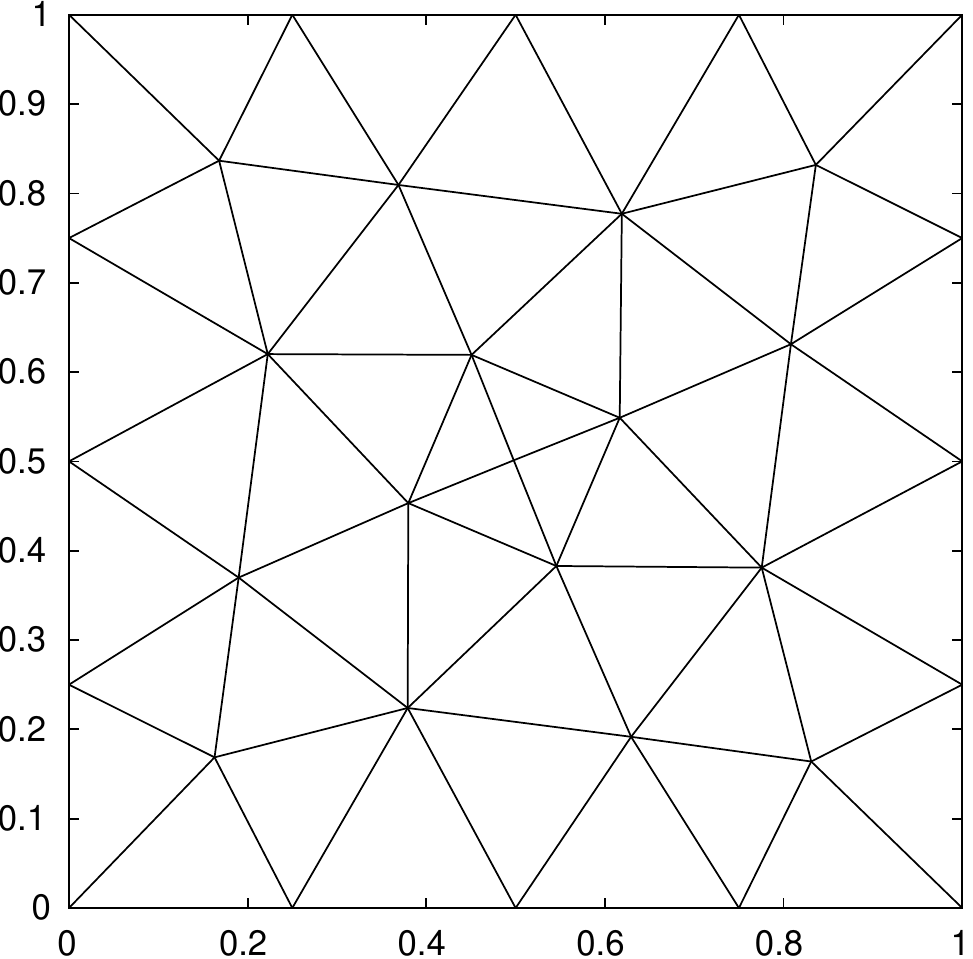}
	\includegraphics[width=0.3\linewidth, height=0.2\textheight]{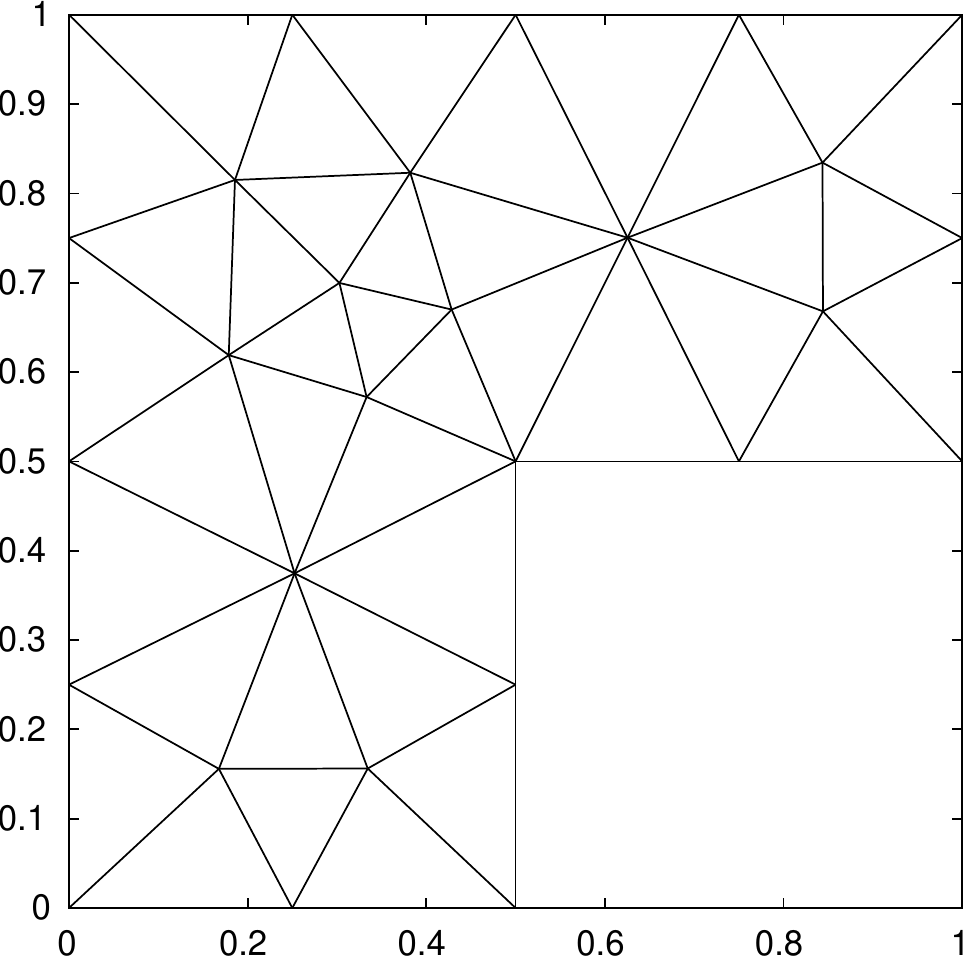}
	\includegraphics[width=0.3\linewidth, height=0.2\textheight]{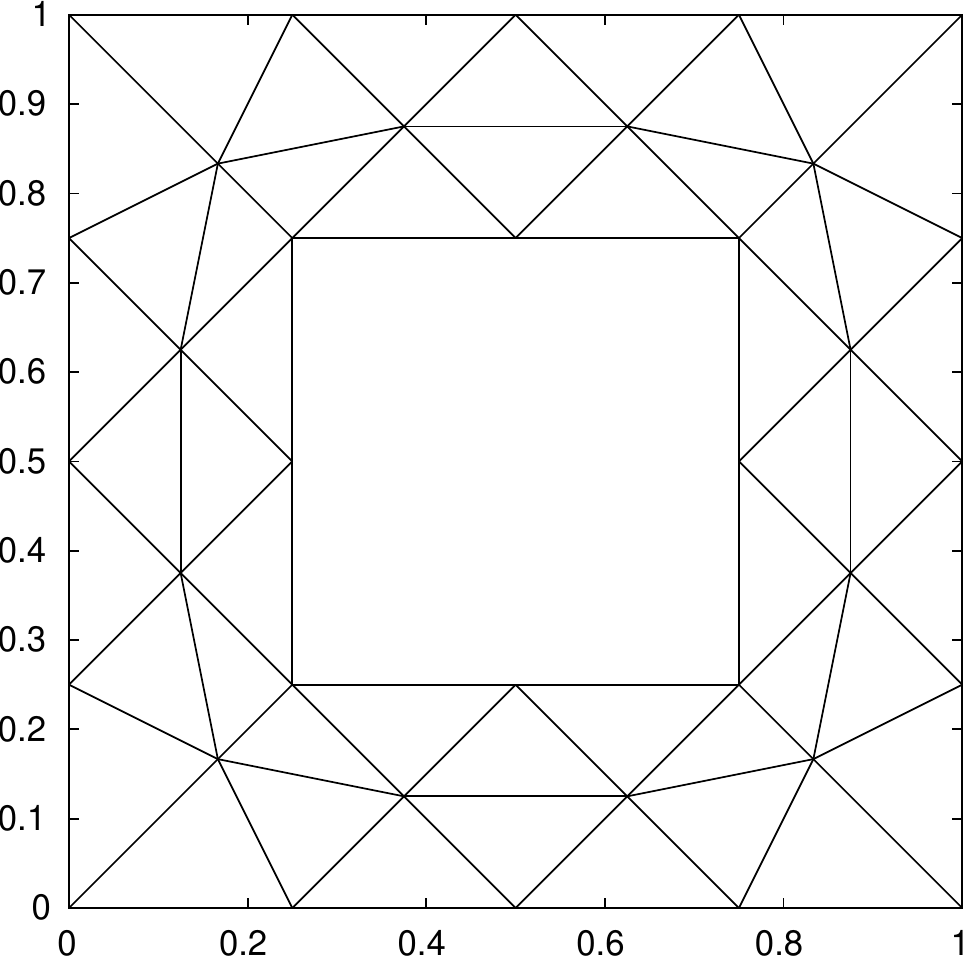}
	\caption{Sample meshes for $\Omega_1$ (left), $\Omega_2$ (middle), and $\Omega_3$ (right).}
	\label{fig1}
\end{figure}

\begin{table}[h]
	\centering
	\caption{The first 5 eigenvalues  of $\Omega_1$  with $k=4$.} \label{tab1}
	\begin{tabular}{cccccc}
		\hline
		$h$    &$\lambda_1^h$&$\lambda_2^h$&$\lambda_3^h$&$\lambda_4^h$&$\lambda_5^h$\\
		\hline
        $1/4$&7.08101988e+02&7.08102390e+02&2.35145718e+03&4.25922492e+03&5.02522026e+03\\
		$1/8$&7.07978763e+02&7.07978786e+02&2.35006082e+03&4.25597055e+03&5.02401495e+03\\
	   $1/16$&7.07971973e+02&7.07971975e+02&2.34999027e+03&4.25582307e+03&5.02399272e+03\\
	   $1/32$&7.07971564e+02&7.07971564e+02&2.34998613e+03&4.25581473e+03&5.02399235e+03\\
	   $1/64$&7.07971528e+02&7.07971555e+02&2.34998587e+03&4.25581421e+03&5.02399235e+03\\
		\hline
	\end{tabular}
\end{table}

\begin{table}[h]
	\centering
	\caption{Convergence rate for $\Omega_1$ with $k=4$ (relative error).} \label{tab2}
	\begin{tabular}{cccc}
		\hline
		$h$    &$\lambda_1^h$& error& order\\
		\hline
		\hline
		$1/4$&7.08101988e+02&1.74021691e-04&-\\
		$1/8$&7.07978763e+02&9.59045415e-06&4.1815\\
	   $1/16$&7.07971973e+02&5.77922813e-07&4.0527\\
	   $1/32$&7.07971564e+02&5.08588883e-08&3.5063 \\
	   $1/64$&7.07971528e+02&-&-\\
		\hline
	\end{tabular}
\end{table}

\begin{table}[h]
	\centering
	\caption{The first 5 eigenvalues of $\Omega_2$ with $k=4$.} \label{tab3}
	\begin{tabular}{cccccc}
		\hline
		$h$    &$\lambda_1^h$&$\lambda_2^h$&$\lambda_3^h$&$\lambda_4^h$&$\lambda_5^h$\\
		\hline
       $1/4$&5.34885649e+02&1.57586875e+03&6.10288551e+03&6.40711482e+03&1.09459861e+04\\
		$1/8$&5.35061810e+02&1.57477474e+03&6.09556539e+03&6.37916246e+03&1.09184358e+04\\
	   $1/16$&5.35222062e+02&1.57468831e+03&6.09528577e+03&6.37104166e+03&1.09152964e+04\\
	   $1/32$&5.35292267e+02&1.57467206e+03&6.09528045e+03&6.36787675e+03&1.09143027e+04\\
	   $1/64$&5.35320748e+02&1.57466664e+03&6.09528434e+03&6.36661570e+03&1.09139180e+04\\
		\hline
	\end{tabular}
\end{table}

\begin{table}[h]
	\centering
	\caption{Convergence rate for $\Omega_2$ with $k=4$ (relative error).} \label{tab4}
	\begin{tabular}{cccc}
		\hline
		$h$    &$\lambda_1^h$& error& order\\
		\hline
		$1/4$&5.34885649e+02&3.29341761e-04&-\\
		$1/8$&5.35061810e+02&2.99502830e-04&0.1370\\
	   $1/16$&5.35222062e+02&1.31169871e-04&1.1911\\
	   $1/32$&5.35292267e+02&5.32057764e-05&1.3018\\
	   $1/64$&5.35320748e+02&-&-\\
		\hline
	\end{tabular}
\end{table}

\begin{table}[h]
	\centering
	\caption{The first 5 eigenvalues of $\Omega_3$ with $k=4$.} \label{tab5}
	\begin{tabular}{cccccc}
		\hline
		$h$    &$\lambda_1^h$&$\lambda_2^h$&$\lambda_3^h$&$\lambda_4^h$&$\lambda_5^h$\\
		\hline
        $1/4$&9.43570924e+02&9.43570924e+02&3.35118080e+03&5.10757870e+03&1.03672699e+04\\
		$1/8$&9.40543704e+02&9.40543704e+02&3.33230800e+03&5.11255084e+03&1.03470233e+04\\
	   $1/16$&9.39507116e+02&9.39507116e+02&3.32612997e+03&5.11519580e+03&1.03445476e+04\\
	   $1/32$&9.39103168e+02&9.39103168e+02&3.32373447e+03&5.11630255e+03&1.03438189e+04\\
	   $1/64$&9.38943028e+02&9.38943036e+02&3.32278551e+03&5.11674950e+03&1.03435487e+04\\
		\hline
	\end{tabular}
\end{table}

\begin{table}[h]
	\centering
	\caption{Convergence rate for $\Omega_3$ with $k=4$ (relative error).} \label{tab6}
	\begin{tabular}{cccc}
		\hline
		$h$    &$\lambda_1^h$& error& order\\
		\hline
		$1/4$&9.43570924e+02&3.20825910e-03&-\\
		$1/8$&9.40543704e+02&1.10211572e-03&1.5415\\
	   $1/16$&9.39507116e+02&4.29957430e-04&1.3580\\
	   $1/32$&9.39103168e+02&1.70524522e-04&1.3342\\
	   $1/64$&9.38943028e+02&-&-\\
		\hline
	\end{tabular}
\end{table}

\subsection{A posteriori error estimates}

Figure \ref{fig_rate} shows global error estimators  and the relative errors of some simple eigenvalues for the three domains.
It can be observed that both the relative errors and the estimators have the same convergence rates.
Figure \ref{fig_dist} shows the distribution of the local estimators. The estimators are large at corners and catch the singularities effectively.

\begin{figure} \centering
\subfigure[the third eigenvalue of $\Omega_1$] { \label{fig:a}
\includegraphics[width=0.4\columnwidth]{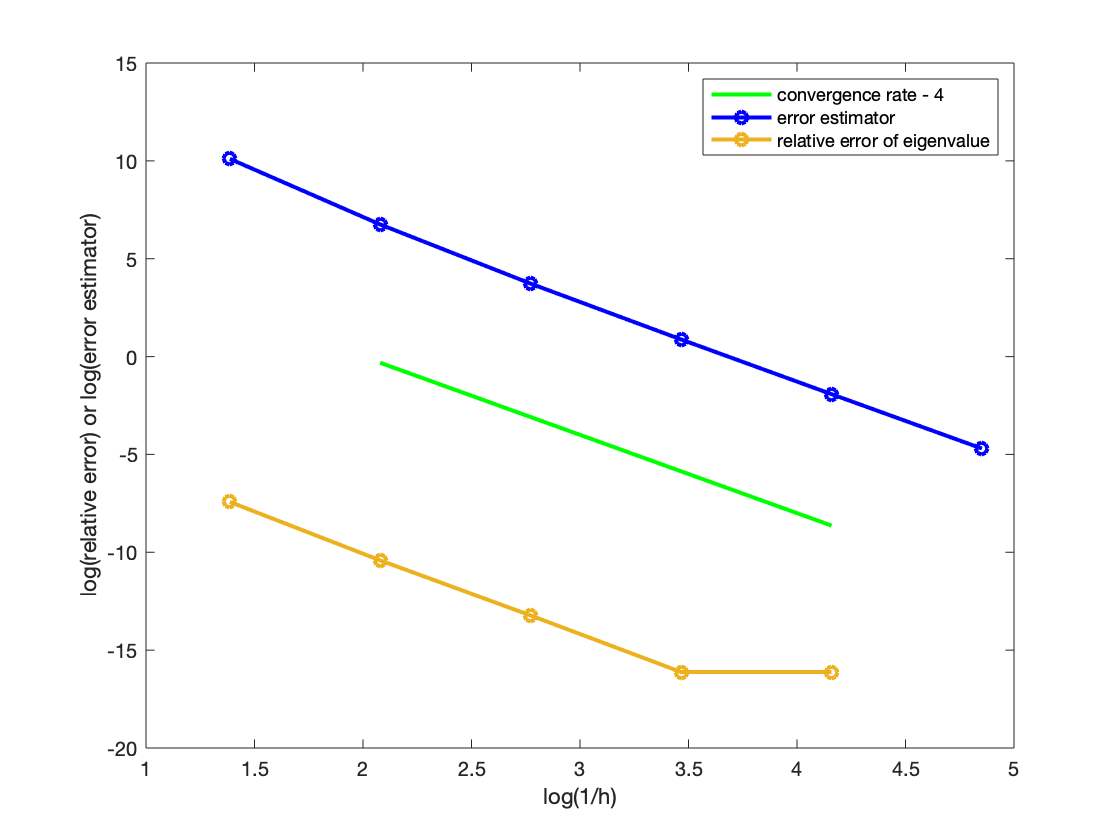}
}
\subfigure[the first eigenvalue on $\Omega_2$] { \label{fig:b}
\includegraphics[width=0.4\columnwidth]{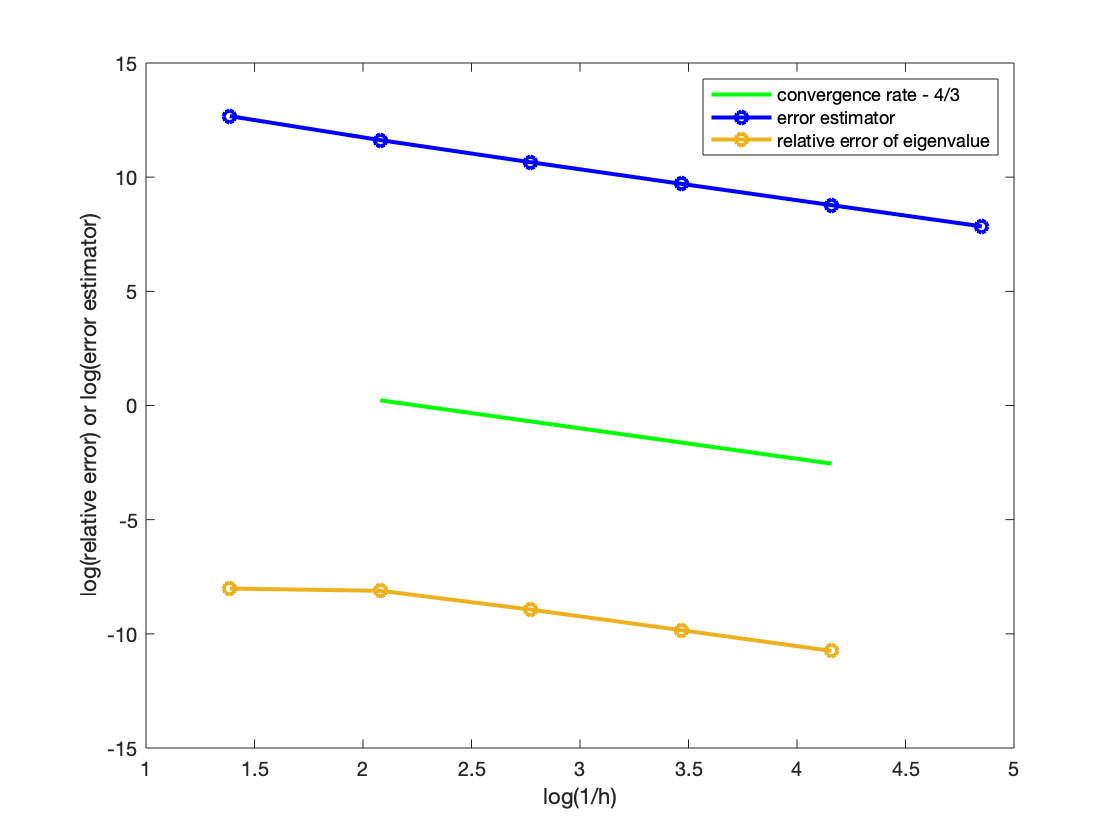}
}
\subfigure[the third eigenvalue on $\Omega_3$] { \label{fig:b}
\includegraphics[width=0.4\columnwidth]{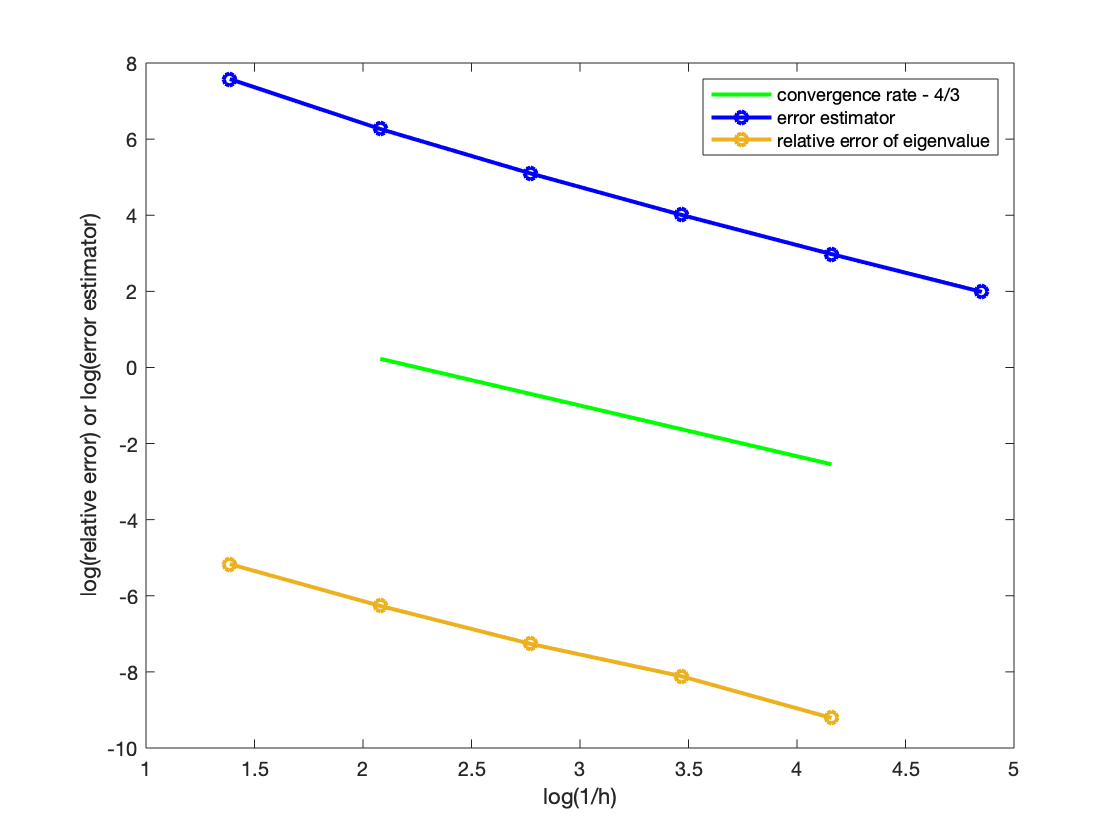}
}
\caption{The convergence rates of error estimators and the relative errors} \label{fig_rate}
\label{fig}
\end{figure}

\begin{figure}\label{fig_dist}
	\includegraphics[scale=0.22]{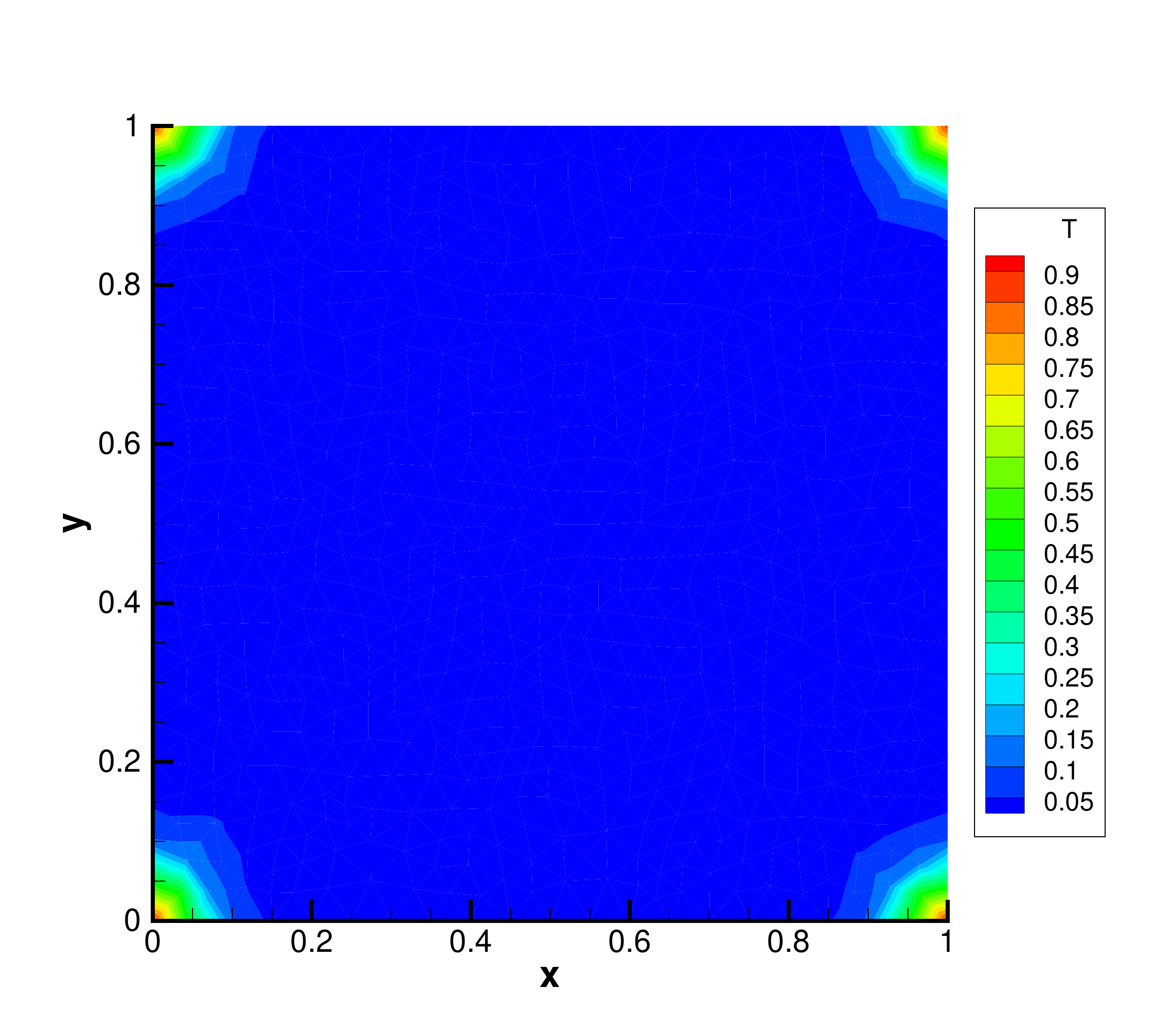}
	\includegraphics[scale=0.22]{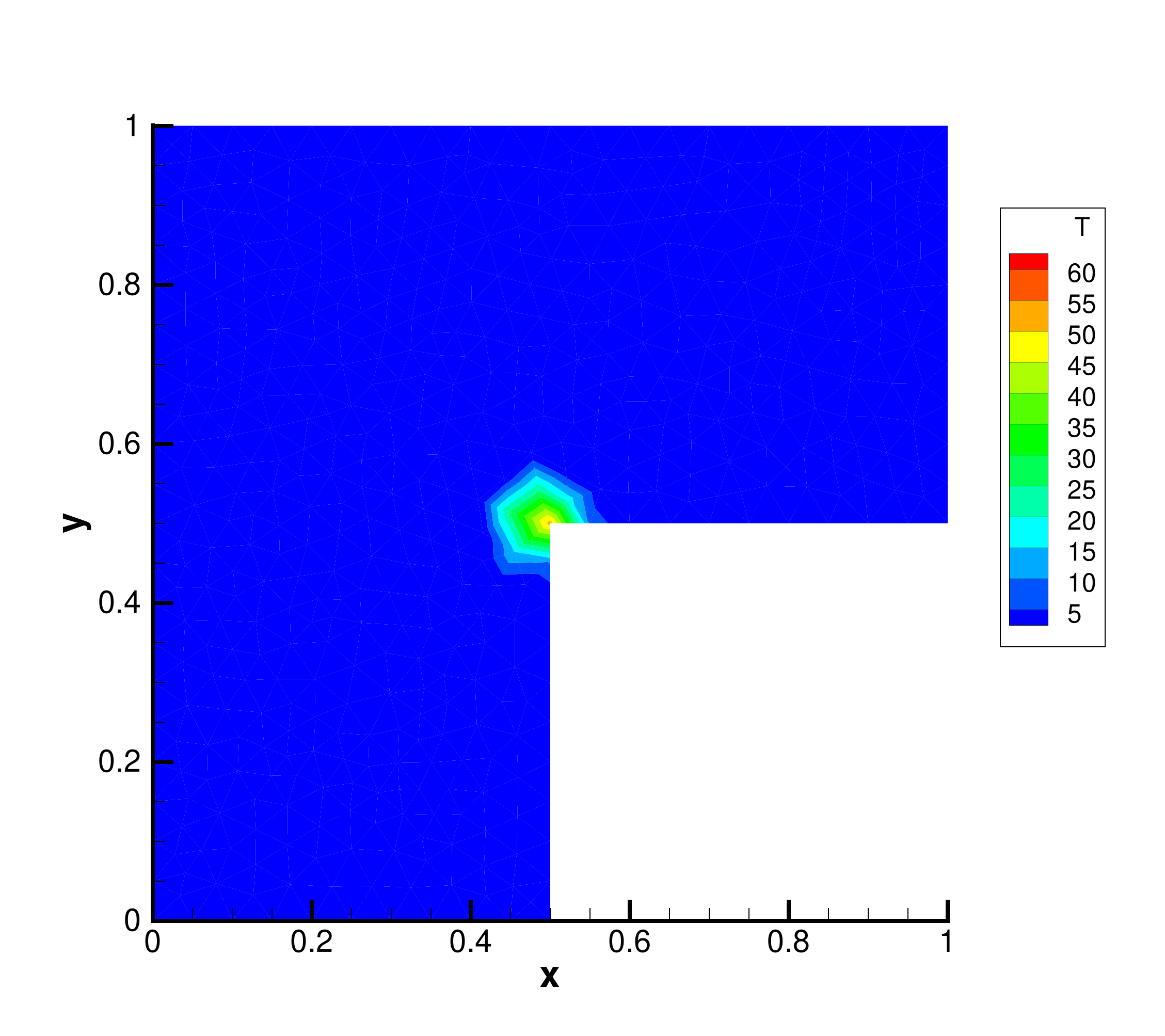}
	\includegraphics[scale=0.22]{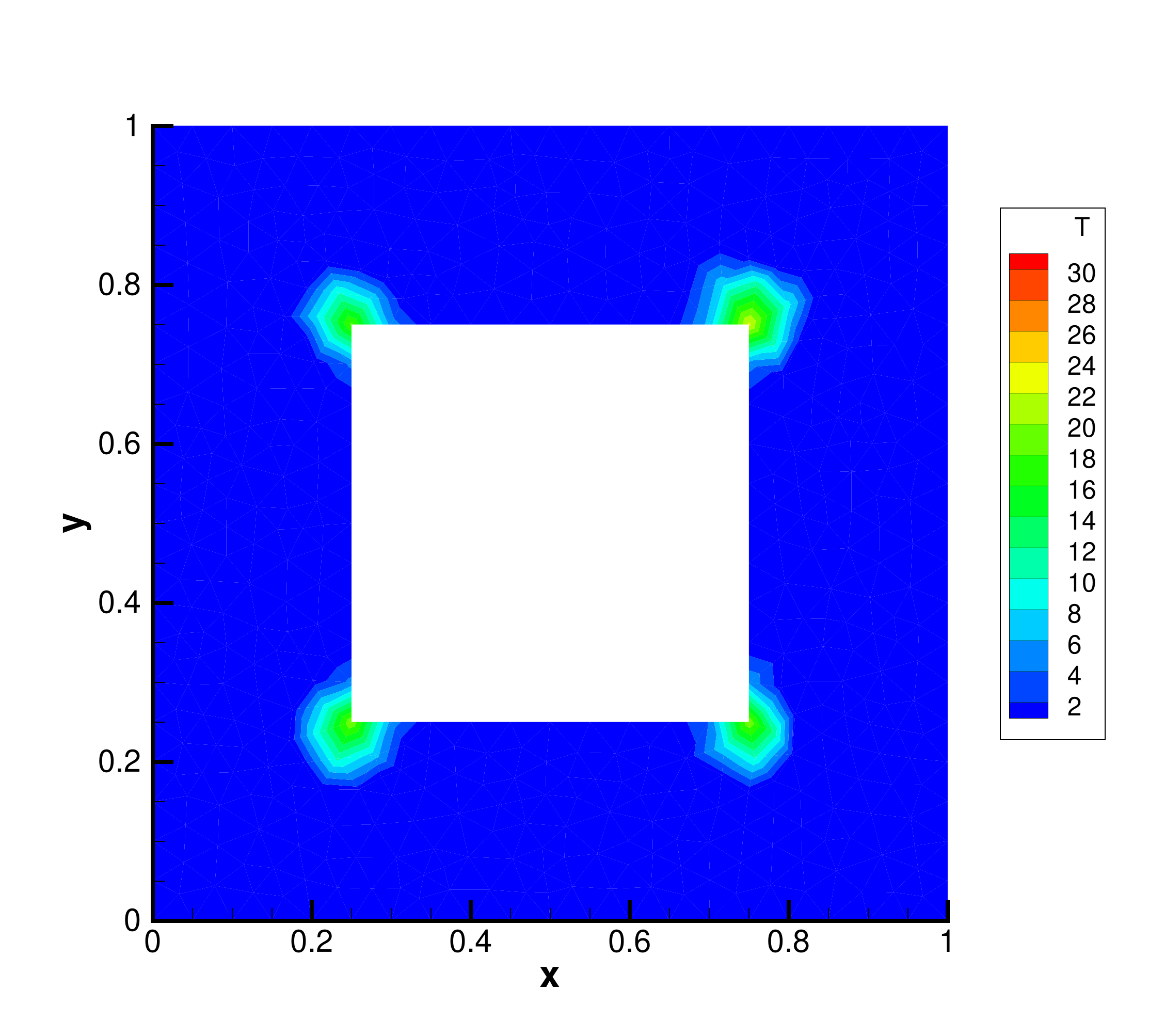}
	\caption{The local estimators.}
	\label{fig_dist}
\end{figure}

\section{Conclusion}
A $H(\tc^2)$-conforming element is proposed for the quad-curl problem in 2D.
We construct a priori and robust a posteriori error estimates for the eigenvalue problem.
Due to a new decomposition for the solution for the quad-curl problem,
the theory assumes no extra regularity of the eigenfunctions. In future, we plan to use the estimator to develop adaptive finite element methods. 
The 3D counterpart is anther interesting but challenging topic.

\bibliographystyle{plain}
\bibliography{QuadCurlEIg2DR1}{}
~\\
\end{document}